\documentclass[a4paper,reqno]{amsart}
\usepackage{graphicx}
\usepackage{mathrsfs}
\usepackage{amsfonts}
\usepackage{amssymb}
\usepackage{amsmath}
\usepackage{amsthm}
\usepackage{amscd}
\usepackage[all,2cell]{xy}
\usepackage[usenames]{color}
\usepackage{hyperref}
\UseAllTwocells \SilentMatrices
\numberwithin{equation}{section}

\usepackage{comment}

\theoremstyle{plain}
\newtheorem{thm}{Theorem}[section]
\newtheorem{prop}[thm]{Proposition}
\newtheorem{cor}[thm]{Corollary}
\newtheorem{lem}[thm]{Lemma}

\theoremstyle{definition}
\newtheorem{defi}[thm]{Definition}
\newtheorem{exm}[thm]{Example}
\theoremstyle{remark}
\newtheorem{rmk}[thm]{\bf Remark}
\newcommand{\D}{\mathbf{D}}
\newcommand{\K}{\mathbf{K}}
\newcommand{\CK}{\mathcal{K}}
\newcommand{\ac}{\mathrm{ac}}

\def\Z{\mathbb{Z}}

\def\Ke{{\rm Ker}}
\def\Im{{\rm Im}}
\def\La{\mathbf{\Lambda}}
\def\I{\mathcal{I}}
\def\P{\mathcal{P}}

\def\M{\mathcal{M}}
\def\s{\sharp}

\def\C{\mathcal{C}}
\def\aa{\alpha}
\def\b{\beta}
\def\z{\zeta}
\def\g{\gamma}
\def\d{\delta}
\def\bu{\bullet}
\def\lra{\longrightarrow}
\def\xra{\xrightarrow[]{}}

\newcommand{\G}{\Gamma}

\renewcommand{\-}{\mbox{-}}
\newcommand{\Inj}{\mathrm{Inj}}
\newcommand{\rad}{\mathrm{rad}}

\newcommand{\Proj}{\mathrm{Proj}}
\newcommand{\Hom}{\mathrm{Hom}}

\newcommand{\Mod}{\mathrm{Mod}}

\begin{document}

\title{The projective Leavitt complex}

\author{Huanhuan Li}



\date{\today}

\begin{abstract} Let $Q$ be a finite quiver without sources, and $A$ be the corresponding radical square zero algebra. We construct an explicit compact generator for the homotopy category of acyclic complexes of projective $A$-modules. We call such a generator the projective Leavitt complex of $Q$. This terminology is justified by the following result: the opposite differential graded endomorphism algebra of the projective Leavitt complex of $Q$ is quasi-isomorphic to the Leavitt path algebra of $Q^{\rm op}$.  Here, $Q^{\rm op}$ is the opposite quiver of $Q$ and the Leavitt path algebra of $Q^{\rm op}$ is naturally $\Z$-graded and viewed as a differential graded algebra with trivial differential.
\end{abstract}

\maketitle

\section{Introduction}

Let $A$ be a finite dimensional algebra over a field $k$. We denote by ${\K_{\ac}}(A\-\Proj)$ the homotopy category of acyclic complexes of projective $A$-modules. This category is a compactly generated triangulated category whose subcategory of compact objects is triangle equivalent to the opposite category of the singularity category \cite{bbu,bo} of the opposite algebra $A^{\rm op}$. 

In this paper, we construct an explicit compact generator for the homotopy category ${\K_{\ac}}(A\-\Proj)$ in the case that $A$ is an algebra with radical square zero. The compact generator is called the \emph{projective Leavitt complex}. Recall from \cite[Theorem 6.2]{bcy} that the homotopy category ${\K_{\ac}}(A\-\Proj)$ was described in terms of Leavitt path algebra in the sense of \cite{bap, bamp}. We prove that the opposite differential graded endomorphism algebra of the projective Leavitt complex of a finite quiver without sources is quasi-isomorphic to the Leavitt path algebra of the opposite quiver. Here, the Leavitt path algebra is naturally $\Z$-graded and viewed as a differential graded algebra with trivial differential.

Let $Q$ be a finite quiver without sources, and $A=kQ/J^2$ be the corresponding algebra with radical square zero. We introduce the projective Leavitt complex $\P^{\bu}$ of $Q$ in Definition \ref{defproj}. Then we prove that $\P^{\bu}$ is acyclic; see Proposition \ref{projacy}. 

Denote by $L_{k}(Q^{\rm op})$ the Leavitt path algebra of $Q^{\rm op}$ over $k$, which is naturally $\Z$-graded. Here, $Q^{\rm op}$ is the opposite quiver of $Q$. We consider $L_k(Q^{\rm op})$ as a differential graded algebra with trivial differential.

The following is the main result of the paper.

\vskip 5pt

\noindent $\mathbf{Theorem.}$ \emph{Let $Q$ be a finite quiver without sources, and $A=kQ/J^{2}$ be the corresponding finite dimensional algebra with  radical square zero.}

\begin{enumerate}
\item[(1)] \emph{The projective Leavitt complex $\mathcal{P}^{\bullet}$ of $Q$ is a
compact generator for the homotopy category $\K_{\rm ac}(A\-\Proj)$.}

\item[(2)] \emph{The opposite differential graded endomorphism algebra of the projective Leavitt complex $\P^{\bu} $ of $Q$ is quasi-isomorphic to the Leavitt path algebra $L_k(Q^{\rm op})$. }\hfill $\square$
\end{enumerate}

The Theorem is a combination of Theorem \ref{tcproj} and Theorem \ref{rightqproj}. We mention that for the construction of the projective Leavitt complex $\P^{\bu}$, we use the basis of the Leavitt path algebra $L_k(Q^{\rm op})$ given by \cite[Theorem 1]{baajz}.

%Recall that the homotopy category ${\K_{\ac}}(A\-\Proj)$ was described in terms of Leavitt path algebras; see \cite[Theorem 6.2]{cy}. The projective Leavitt complex, as a compact generator of this homotopy category, may have connection with the Leavitt path algebra. 

%We mention that for the construction of the projective Leavitt complex $\P^{\bu}$, we use the basis of the Leavitt path algebra $L_k(Q^{\rm op})$ given by \cite[Theorem 1]{baajz}.

For the proof of $(1)$, we construct subcomplexes of $\P^{\bu}$. For $(2)$, we actually prove that the projective Leavitt complex has the structure of a differential graded $A$-$L_k(Q^{\rm op})$-bimodule, which is right quasi-balanced. Here, we consider $A$ as a differential graded algebra concentrated in degree zero, and $L_k(Q^{\rm op})$ is naturally $\Z$-graded and viewed as a differential graded algebra with trivial differential.

The paper is structured as follows.
In section $\ref{section2}$, we introduce the projective Leavitt complex $\P^{\bu}$ of $Q$ and prove that it is acyclic. In section $\ref{section3}$, we recall some notation and
prove that the projective Leavitt complex $\P^{\bu}$ is a compact generator of
the homotopy category of acyclic complexes of projective 
$A$-modules.
In section $\ref{section4}$, we recall some facts of the Leavitt path algebra and 
endow the projective Leavitt complex $\P^{\bu}$ with a differential graded $L_k(Q^{\rm op})$-module structure, which makes it become an $A$-$L_k(Q^{\rm op})$-bimodule.
In section $\ref{section5}$, we prove that the opposite differential graded endomorphism algebra of  $\P^{\bu}$ is quasi-isomorphic to the Leavitt path algebra $L_k(Q^{\rm op})$.

\section{The projective Leavitt complex of a finite quiver without sources}
\label{section2}

In this section, we introduce the projective Leavitt complex of a finite quiver without sources, and prove that it is an acyclic complex of projective modules over the corresponding finite dimensional algebra with radical square zero.

\subsection{The projective Leavitt complex}

A quiver $Q=(Q_{0}, Q_{1}; s, t)$
consists of a set $Q_{0}$ of vertices, a set $Q_{1}$ of arrows and
two maps $s, t: Q_{1}\xrightarrow []{}Q_{0}$,
which associate to each arrow $\alpha$ its starting vertex $s(\alpha)$ and its terminating
vertex $t(\alpha)$, respectively. A quiver $Q$
is finite if both the sets $Q_{0}$ and $Q_{1}$ are finite.

A path in the quiver $Q$
is a sequence $p=\alpha_{m}\cdots\alpha_{2}\alpha_{1}$ of arrows with
$t(\alpha_{k})=s(\alpha_{k+1})$ for $1\leq k\leq m-1$.
The length of $p$, denoted by $l(p)$, is $m$. The starting vertex of $p$, denoted by $s(p)$,
is $s(\alpha_{1})$.
The terminating vertex of $p$, denoted by $t(p)$, is $t(\alpha_{m})$.
We identify an
arrow with a path of length one. For each vertex $i\in Q_{0}$, we associate to it a trivial path $e_{i}$ of length zero. Set $s(e_i)=i=t(e_i)$.
Denote by $Q_{m}$ the set of all paths in $Q$ of length $m$ for each $m\geq 0$.

Recall that a vertex of $Q$ is a sink if there is no arrow starting at it and a source if there is no arrow terminating at it. Recall that for a vertex $i$ which is not a sink, we can choose an arrow $\b$ with $s(\b)=i$, which is called the \emph{special arrow} starting at vertex $i$; see \cite{baajz}. For a vertex $i$ which is not a source, fix an arrow $\gamma$
with $t(\gamma)=i$. We call the fixed arrow the \emph{associated arrow} terminating at $i$.
For an associated arrow $\alpha$,
we set \begin{equation} \label{eq:vv}
T(\alpha)=\{\beta\in Q_{1}\;| \;t(\beta)=t(\alpha), \beta\neq\alpha\}.
\end{equation}

\begin{defi}
For two paths $p=\aa_{m}\cdots\aa_{2}\aa_{1}$ and $q=\b_{n}\cdots \b_{2}\b_{1}$ with $m,n\geq 1$, we call the pair $(p, q)$ an
\emph{associated pair} in $Q$
if $s(p)=s(q)$, and either $\alpha_{1}\neq\beta_{1}$, or  $\alpha_{1}=\beta_{1}$ is not associated.  In addition, we call $(p, e_{s(p)})$ and $(e_{s(p)}, p)$ \emph{associated pairs} in $Q$ for each path $p$ in $Q$. \hfill $\square$
\end{defi}

For each vertex $i\in Q_{0}$ and $l\in \Z$, set
\begin{equation}
\label{p}
\mathbf{\Lambda}^{l}_{i}=\{(p, q)\;|\;(p, q) \text{~is an associated pair with~}l(q)-l(p)=l \text{~and~} t(p)=i\}.
\end{equation}

\begin{lem} Let $Q$ be a finite quiver without sources. The above set $\mathbf{\Lambda}^{l}_{i}$
is not empty for each vertex $i$ and each integer $l$.
\end{lem}

\begin{proof} Recall that the opposite quiver $Q^{\rm op}$ of the quiver $Q$ has arrows with opposite directions. For each vertex $i\in Q_0$, fix the special arrow of $Q^{\rm op}$ starting at $i$ as the opposite arrow of the associated arrow of $Q$ terminating at $i$. Observe that for each vertex $i$ and each integer $l$, $\mathbf{\Lambda}^{l}_{i}$ is one-to-one corresponded to
$\{(q^{\rm op}, p^{\rm op})\;\;|\; \;(q^{\rm op}, p^{\rm op})\text{~~is an admissible pair in~~} Q^{\rm op} \text{~~with}~~l(p^{\rm op})-
l(q^{\rm op})=-l\text{~~~~~~and}\\~s(p^{\rm op})=i\}.$ Here, refer to \cite[Definition 2.1]{bli} for the definition of admissible pair. By \cite[Lemma 2.2]{bli}, the later set is not empty. The proof is completed.
\end{proof}

Let $k$ be a field and $Q$ be a finite quiver. For each $m\geq 0$, denote by $kQ_{m}$ the $k$-vector space with basis $Q_{m}$. The path algebra $kQ$ of the quiver $Q$ is defined as $kQ=\bigoplus_{m\geq 0}kQ_{m}$, whose multiplication is given by such that for two paths $p$ and $q$ in $Q$, if $s(p)=t(q)$, then the product $pq$ is their concatenation;
otherwise, we set the product $pq$ to be zero. Here, we write the concatenation of paths
from right to left.

We observe that for any path $p$ and vertex $i\in Q_0$, $p e_{i}=\delta_{i, s(p)}p$ and
$e_ip=\delta_{i, t(p)}p$, where $\delta$ denotes the Kronecker symbol. We have that $\sum_{i\in Q_{0}}e_{i}$ is the unit of the path algebra $kQ$. Denote by $J$ the two-sided ideal of $kQ$ generated by arrows.

We consider the quotient algebra $A=kQ/J^{2}$;
it is a finite dimensional algebra with radical square zero. Indeed, $A=kQ_{0}\oplus kQ_{1}$
as a $k$-vector space. The Jacobson radical of $A$ is $\rad A=kQ_{1}$ satisfying $(\rad A)^{2}=0$. 
For each vertex $i\in Q_0$ and each arrow $\aa\in Q_1$, we abuse $e_i$ and $\aa$ with their canonical images in the algebra $A$.

Denote by $P_{i}=Ae_{i}$ the indecomposable projective left $A$-module for $i\in Q_{0}$. We have the following observation.

\begin{lem} \label{lemma-projmap} Let $i, j$ be two vertices in $Q$, and $f: P_i\xrightarrow[]{} P_j$ be a $k$-linear map. Then $f$ is a left $A$-module morphism if and only if 
\begin{equation*}
\begin{cases}
f(e_i)=\delta_{i, j}\lambda e_j+\sum\limits_{\{\b\in Q_{1}\;|\; s(\b)=j,\;t(\b)=i\}}\mu(\b)\b\\
f(\aa)=\delta_{i, j}\lambda\aa
\end{cases}
\end{equation*} 
with $\lambda$ and $\mu(\b)$ scalars for all $\alpha\in Q_1$
with $s(\aa)=i$.  \hfill $\square$
\end{lem}

For a set $X$ and an $A$-module $M$, the coproduct $M^{(X)}$ will be
understood as
$\bigoplus_{x\in X}M\zeta_{x}$ such that each component $M\zeta_{x}$ is $M$.
For an element $m\in M$, we use $m\zeta_{x}$
to denote the corresponding element in the component $M\zeta_{x}$.

For a path $p=\alpha_{m}\cdots\aa_2\alpha_{1}$ in $Q$ of length
$m\geq 2$, we denote by $\widehat{p}=\alpha_{m-1}\cdots\alpha_{1}$ and $\widetilde{p}=\alpha_{m}\cdots\alpha_{2}$ the two \emph{truncations} of
$p$. For an arrow $\aa$, denote by $\widehat{\aa}=e_{s(\aa)}$ and $\widetilde{\aa}=e_{t(\aa)}$.

\begin{defi} \label{defproj} Let $Q$ be a finite quiver without sources. The \emph{projective Leavitt complex} $\mathcal{P}^{\bullet}=(\mathcal{P}^{l}, \delta^{l})_{l\in\Z}$ of $Q$ is defined as follows:

\noindent $(1)$ the $l$-th component $\mathcal{P}^{l}=\bigoplus_{i\in Q_{0}}{P_{i}}^{(\mathbf{\La}^{l}_{i})}$.

\noindent $(2)$ the differential $\delta^{l}: \mathcal{P}^{l}\longrightarrow \mathcal{P}^{l+1}$ is given by  $\delta^{l}(\alpha\zeta_{(p, q)})=0$ and \begin{equation*}\delta^{l}(e_{i}\zeta_{(p, q)})=\begin{cases}
 \beta\zeta_{(\widehat{p},q)}, & \text{if $p=\beta\widehat{p}$};\\
\sum_{\{\beta\in Q_{1}\;|\;  t(\beta)=i\}}\beta\zeta_{(e_{s(\beta)}, q\beta)},& \text{if $l(p)=0$,}
\end{cases}
\end{equation*}
for any $i\in Q_{0}$, $(p, q)\in \mathbf{\La}^{l}_{i}$
and $\alpha\in Q_{1}$ with $s(\aa)=i$. \hfill $\square$
\end{defi}

Each component $\mathcal{P}^{l}$ is a projective $A$-module. The differential $\d^{l}$ are $A$-module morphisms; compare Lemma \ref{lemma-projmap}. By the definition of the differential $\d^{l}$, it is direct to see that $\d^{l+1}\circ \d^{l}=0$ for each $l\in \Z$. In conclusion, $\mathcal{P}^{\bullet} $ is a complex of projective $A$-modules. 

\begin{comment}
The following fact is immediate.

\begin{lem} \label{delta} For $l\in\Z$, $i\in Q_{0}$ and $(p, q)\in\La^l_i$ with $l(q)=0$,
we have $\b \z_{(p, q)}=\d^{l-1}(e_{t(\b)}\z_{(\b p, q)})$
for $\b\in Q_{1}$ with $s(\b)=i$.
\end{lem}

The definition of the projective Leavitt complex of $Q$ does not depend on the choice of designated arrows of $Q$;
see Proposition $\ref{unique}$ $(1)$ in section $4$.
\end{comment}

\subsection{The acyclicity of the projective Leavitt complex} We will show that the projective Leavitt complex is acyclic. 

In what follows, $V, V'$ are two $k$-vector spaces and $f: V\xra V'$ is a $k$-linear map.
Suppose that $B$ and $B'$ are $k$-bases of $V$ and $V'$, respectively. We say that the triple $(f, B, B')$ satisfies \emph{Condition} (X) if $f(B)\subseteq B'$ and  the restriction of $f$ on $B$ is injective. In this case, we have $\Ke f=0$.

We suppose further that there  are disjoint unions $B=B_{0}\cup B_{1}\cup B_{2}$
and $B'=B'_{0}\cup B'_{1}$.
%Assume that if $B'_0=\emptyset$, then $B_2=\emptyset$. 
We say that the triple $(f, B, B')$ satisfies \emph{Condition} (W) if the following statements hold:
\begin{enumerate}
\item[(W1)] $f(b)=0$ for each $b\in B_{0}$;

\item[(W2)] $f(B_{1})\subseteq B_{1}'$ and $(f_{1}, B_{1}, B')$ satisfies Condition (X) where $f_{1}$ is the restriction of $f$ to the subspace spanned by $B_{1}$.

\item[(W3)] For $b\in B_{2}$, $f(b)=b_{0}+\sum_{c\in B_{1}(b)}f(c)$ for some $b_{0}\in B_{0}'$ and some finite subset $B_{1}(b)\subseteq B_{1}$. Moreover, if $b, b'\in B_{2}$ and $b\neq b'$, then $b_{0}\neq b_{0}'$.
\end{enumerate}

We have the following observation. The proof of it is similar as that of \cite[Lemma 2.7]{bli}. We omit it here.

\begin{lem} \label{kerimproj1} Assume that $(f, B, B')$ satisfies Condition (W),
then $B_{0}$ is a $k$-basis of $\Ke f$ and $f(B_{1})\cup \{b_{0}\;| \;b\in B_2\}$ is a $k$-basis of 
$\Im f$.\hfill $\square$
\end{lem}

\begin{comment}
\begin{proof} Suppose that $f(x)=0$ for $x=\sum_{b\in B_{1}}\mu_{b}b+\sum_{b\in B_{2}}\nu_{b}b$ with
$\mu_{b}$ and $\nu_{b}$ scalars.
We have $$f(x)=\sum_{b\in B_{1}}\mu_{b}f(b)+\sum_{b\in B_{2}}\nu_{b}(b_{0}+\sum_{y\in B_{1}(b)}f(y))=0.$$
For (W3), we have $\nu_{b}=0$.
By (W2), we have $\sum_{b\in B_{1}}\mu_{b}b=0$. 
Hence $x=0$. This implies that $B_{0}$ is a $k$-basis of $\Ke f$ as $k$-vector space.
Observe that $b_{0}\in \Im f$ for each $b\in B_{2}$. By (W1) and (W3), each element in $\Im f$ is a $k$-linear combinations from $f(B_1)$ and $\{b_{0}\;|\; b\in B_2\}$. Recall that $f(B_{1})\subseteq B'_{1}$ and $\{b_{0}\;|\; b\in B_2\}\subseteq B'_{0}$. This proves that $f(B_{1})\cup \{b_{0}\;|\; b\in B_2\}$ is a $k$-basis of $\Im f$. 
\end{proof}
\end{comment}

From now on, $Q$ is a finite quiver without sources. We consider the differential $\d^l: \P^l\xra \P^{l+1}$ in Definition \ref{defproj}. We have the following $k$-basis of $\P^l$: $$B^l=\{e_i\z_{(p, q)}, \aa\z_{(p, q)}\; |\; i\in Q_0, (p, q)\in \La^l_i \text{~and~} \aa\in Q_1 \text{~with~} s(\aa)=i\}.$$
Denote by $B^{l}_{0}=\{\alpha\zeta_{(p, q)}|~i
\in Q_{0},(p, q)\in \mathbf{\La}_{i}^{l} ~and~\alpha\in Q_{1} ~with~ s(\aa)=i\}$ a subset of $B^l$. Set $$B^l_2=\{e_i\z_{(e_i, q)}\;|\; i\in Q_0, (e_i, q)\in \La^l_i\}$$ for $l\geq 0$. If $l<0$, put $B^l_2=\emptyset$. Take $B^l_1=B^l\setminus (B^l_0\cup B^l_2)$. Then we have the disjoint union $B^l=B^l_0\cup B^l_1\cup B^l_2$. Set $B'^{l}_{0}=\{\b\zeta_{(e_{s(q)}, q)}\;|\;i\in Q_{0}, (e_{s(q)}, q)\in \La_{i}^{l}~\text{such that}~q
=\widetilde{q}\b ~and~\b \text{~is associated}\}$ for $l\in \Z$. 
We mention that $B'^{l}_{0}=\emptyset$ for $l<0$. Take $B'^{l}_{1}=B^{l}\setminus B'^{l}_{0}$ for $l\in \Z$. Then we have the disjoint union $B^{l}=B'^{l}_{0}\cup B'^{l}_{1}$ for each $l\in \Z$.

\begin{lem} \label{deltaker} For each $l\in\mathbb{Z}$, the set $B^{l}_{0}$
is a $k$-basis of ${\Ke}\d^{l}$ and the set $B^{l+1}_0$ is a $k$-basis of $\Im \d^l$.
\end{lem}
\begin{proof} For $l<0$, we have $B^{l}_{2}=\emptyset=B'^{l+1}_{0}$. We observe that the triple $(\d^l, B^l,B^{l+1})$ satisfies Condition (W). Indeed, $\d^l(b)=0$ for each $b\in B^l_0$. The differential $\d^l$ induces an injective map $\d^l: B^l_1\xra B'^{l+1}_1$. Then (W1) and (W2) hold. To see (W3), 
for $l\geq 0$ and each $i\in Q_{0}$, $e_i\zeta_{(e_i, q)}\in B^{l}_{2}$, we have 
$$\d^{l}(e_i\zeta_{(e_i, q)})=\alpha\zeta_{(e_{s(\alpha)}, q\aa)}+\sum_{\beta\in T(\alpha)}\d^{l}(e_{s(\b)}\zeta_{(\beta, q\b)}),$$ where $\aa\in Q_1$ such that $t(\aa)=i$ and $\aa$ is associated. Here, recall $T(\aa)$ from \eqref{eq:vv}. Thus $(e_i\zeta_{(e_i, q)})_0=\alpha\zeta_{(e_{s(\alpha)}, q\aa)}$ and the finite subset $B^l_1(e_i\zeta_{(e_i, q)})=\{e_{s(\b)}\zeta_{(\beta, q\b)}\;|\; \b\in T(\aa)\}$.

Recall that $B^{l+1}_0=\{\alpha\zeta_{(p, q)}|~i
\in Q_{0},(p, q)\in \mathbf{\La}_{i}^{l+1} ~and~\alpha\in Q_{1} ~with~ s(\aa)=i\}$. Now we prove that $B^{l+1}_0=\d^l(B^l_1)\cup \{b_0\;|\; b\in B^l_2\}$. We mention that the set $\{b_0\;|\; b\in B^l_2\}=\{\aa\z_{(e_{s(\aa)}, q\aa)}\;|\;q\in Q_l \text{~and~}\aa \text{~is associated with~} t(\aa)=s(q) \}$. Clearly, $\d^l(B^l_1)\cup \{b_0\;|\; b\in B^l_2\}\subseteq B^{l+1}_0$. Conversely, for each $i\in Q_{0}$ and $(p, q)\in \La^{l+1}_i$, we have 
$\alpha\zeta_{(p, q)}=\delta^{l}(e_{t(\alpha)}\zeta_{(\alpha p, q)})\in \d^l(B^{l}_{1})$ for $\aa\in Q_1$ with $s(\aa)=i$ but
$\alpha\zeta_{(p, q)}\notin \{b_0\;|\; b\in B^l_2\}$. Applying Lemma $\ref{kerimproj1}$ for the triple $(\d^l, B^l, B^{l+1})$, we are done. 
\end{proof}

\begin{prop} \label{projacy} Let $Q$ be a finite quiver without sources. Then the projective Leavitt complex $\mathcal{P}^{\bullet}$
of $Q$ is an acyclic complex.
\end{prop}
\begin{proof} The statement follows directly from Lemma $\ref{deltaker}$.
\end{proof}

\begin{exm} \label{oneandoneproj}Let $Q$ be the following quiver with one vertex and one loop. $$\xymatrix{\scriptstyle{1}\cdot\ar@(ur, dr)[]|{\alpha}}$$
The unique arrow $\aa$ is associated. Set $e=e_1$ and $\mathbf{\La}^{l}=\La^l_1$
for each $l\in \mathbb{Z}$.
It follows that
\begin{equation*}\La^l=\begin{cases}
 \{(\aa^{-l}, e)\}, &
 \text{if $l<0$};\\
  \{(e, e)\},& \text{if $l=0$};\\
   \{(e, \aa^{l})\},&\text{if $l>0$}.
\end{cases}
\end{equation*}

The corresponding algebra $A$ with radical square zero
is isomorphic to $k[x]/(x^2)$. Write $A^{(\La^l)}=A\z^l$, where $\z^l=\z_{(\aa^{-l}, e)}$ for
$l<0$, $\z^0=\z_{(e, e)}$ and $\z^l=\z_{(e, \aa^l)}$ for $l>0$.
Then the projective Leavitt complex $\P^{\bullet}$ of $Q$ is as follows
$$
\cdots
\stackrel{}{\longrightarrow} A\z^{l-1}
\stackrel{\d^{l-1}}{\longrightarrow} A\z^{l}
\stackrel{\d^{l}}{\longrightarrow} A\z^{l+1}
\stackrel{}{\longrightarrow} \cdots,
$$ where the differential $\d^l$ is given by $\d^{l}(e\zeta^l)=\aa\z^{l+1}$
and $\d^{l}(\aa\zeta^l)=0$ for each $l\in \Z$.

Observe that $A$ is a self-injective algebra. The projective Leavitt complex $\P^{\bu}$ is isomorphic to the injective Leavitt complex $\I^{\bu}$ as complexes; compare \cite[Example 2.11]{bli}.
%is isomorphic to a complete resolution of the simple $A$-module; see \cite[Definition 3.1.1]{bu} and compare \cite[Proposition 2.20]{th}.
\end{exm}

\begin{exm} \label{exmproj}  Let $Q$ be the following quiver with one vertex and two loops. $$\xymatrix{\scriptstyle{1}\cdot \ar@(ul, dl)[]|{\alpha_{1}} \ar@(ur, dr)[]|{\alpha_{2}}}$$
We choose $\alpha_{1}$ to be the associated arrow terminating at the unique vertex. Set $e=e_1$ and 
$\mathbf{\La}^{l}=\La^l_1$ for each $l\in \mathbb{Z}$. A pair $(p, q)$ of paths lies in $\La^l$ if and only if $l(q)-l(p)=l$ and $p, q$ do not start with $\aa_1$ simultaneously. 

We denote by $A$ the corresponding radical square zero algebra. The projective Leavitt complex $\mathcal{P}^{\bullet}$ of $Q$ is as follows:

$$\qquad\xymatrix@R=6pt{
   \cdots \ar[r]^{\delta^{-1}} & A^{({\mathbf{\La}^{0}})}\ar[r]^{\delta^{0}} & A^{(\mathbf{\La}^{1})}\ar[r]^{\delta^{1}} &\cdots } $$ We write the differential $\d^0$ explicitely: $\d^0(\aa_k\z_{(p,q)})=0$ and \begin{equation*}\d^0(e\z_{(p,q)})=\begin{cases}
 \aa_k\zeta_{(\widehat{p}, q)}, & \text{if $p=\aa_k\widehat{p}$};\\
 \alpha_{1}\zeta_{(e, q\alpha_{1})}+\alpha_{2}\zeta_{(e, q\alpha_{2})},& \text{if $p=e$}.
\end{cases}
\end{equation*} for $k=1,2$ and $(p,q)\in \La^0$. 
\end{exm}

\section{The projective Leavitt complex as a compact generator}
\label{section3}

In this section, we prove that the projective Leavitt complex is a compact generator of
the homotopy category of acyclic complexes of projective $A$-modules.

\subsection{A cokernel complex and its decomposition}
Let $Q$ be a finite quiver without sources and $A$ be the corresponding algebra with radical square zero. For each $i\in Q_{0}$, $l\in\Z$ and $n\geq 0$,
denote by
$$\mathbf{\Lambda}^{l, n}_{i}=\{(p, q)~|~(p, q)\in \mathbf{\Lambda}^{l}_{i} \text{~with~} p\in Q_{n} \}.$$ Refer to $(\ref{p})$ for the definition of the set $\La^l_i$.

Recall the projective Leavitt complex
$\P^{\bullet}=(\P^{l}, \d^{l})_{l\in \Z}$ of $Q$. For each $l\geq 0$, we denote by $\CK^l=\bigoplus_{i\in Q_{0}}P_{i}^{(\La^{l, 0}_{i})}\subseteq \P^l$, where $P_i=Ae_i$. Observe that the differential $\d^l:\P^l\xra \P^{l+1}$ satisfies $\d^l(\CK^l)\subseteq \CK^{l+1}$. Then we have a subcomplex $\CK^{\bu}$ of $\P^{\bu}$, whose components $\CK^l=0$ for $l<0$. Denote by $\phi^{\bullet}=(\phi^{l})_{l\in \Z}: \CK^{\bullet}\longrightarrow \P^{\bullet}$
be the inclusion chain map by setting $\phi^{l}=0$ for $l<0$. We set $\C^{\bu}$ to be the cokernel of $\phi^{\bu}$.

\begin{comment}
We denote by $\CK^{\bullet}=\bigoplus_{i\geq 0} \CK^{i}$ the following bounded-below complex
$$0\rightarrow\bigoplus_{j\in Q_{0}}P_{j}^{(\La^{0, 0}_{j})}\stackrel{\d^{0}}{\longrightarrow}\cdots\stackrel{\d^{i-1}}{\longrightarrow}\bigoplus_{j\in Q_{0}}P_{j}^{(\La^{i, i}_{j})}\stackrel{\d^{i}}{\longrightarrow}
\bigoplus_{j\in Q_{0}}P_{j}^{(\La^{i+1, i+1}_{j})}
\stackrel{\d^{i+1}}{\longrightarrow} \cdots$$ of $A$-modules where the differential $\d^{i}$ for
$i\geq 0$ is the differential of $\P^{\bullet}$.
\end{comment}

We now describe the cokernel $\C^{\bu}=(\C^l, \widetilde{\d}^l)$ of $\phi^{\bu}$. 
For each vertex $i\in Q_{0}$ and $l\in \Z$,
set $$\La^{l, +}_i=\bigcup_{n>0}\La^{l, n}_{i}.$$ Observe that we have the disjoint union $\La^l_i=\La^{l, 0}_i\cup \La^{l, +}_i$ for $l\geq 0$ and $\La^l_i=\La^{l, +}_i$ for $l<0$. The component of $\C^{\bu}$ is $\C^l=\bigoplus_{i\in Q_{0}}P_{i}^{({\La^{l,+}_{i}})}$ for each $l\in\Z$. We have $\C^l=\P^l$ for $l<0$ and the differential $\widetilde{\d}^l=\d^l$ for $l\leq -2$. The differential $\widetilde{\d}^l:\C^l\xra \C^{l+1}$ for $l\geq -1$ is given as follows:  $\widetilde{{\d}^{l}}(\aa\zeta_{(p, q)})=0$ and
\begin{equation*}\widetilde{{\d}^{l}}(e_i\zeta_{(p, q)})=\begin{cases}
0, &\text{if $l(p)=1$;}\\
\d^l(e_{i}\zeta_{(p, q)}),& \text{otherwise},
\end{cases}
\end{equation*} for any $i\in Q_{0}$, $(p, q)\in \La^{l, +}_{i}$ and $\alpha\in Q_{1}$ with $s(\alpha)=i$. The restriction of $\widetilde{\d}^l$ to $\bigoplus_{i\in Q_0} P_i^{(\La^{l,1}_i)}$ is zero for $l\geq -1$. We emphasize that the differentials $\widetilde{\d}^l$ for $l\geq -1$ are induced by the differentials $\d^l$ in Definition \ref{defproj}.
%The restriction map of $\widetilde{\d^i}$ on
%$\bigoplus_{j\in Q_0}P_j^{(\La^{i, i+1}_j)}$ is zero for $i\geq -1$.

\begin{comment}
Since $\phi^{\bullet}:\CK^{\bullet}\longrightarrow \P^{\bullet}$ is an injective chain map,
we have the following short exact sequence $\xymatrix@C=0.5cm{
  0 \ar[r] & \CK^{\bu} \ar[rr]^{\phi^{\bu}} && \P^{\bu} \ar[rr]^{} && {\rm Coker}(\phi^{\bullet})\ar[r]^{}& 0 }$ of complexes:
\[\resizebox{1\hsize}{!}{$
\begin{CD}
\quad\qquad\CK^{\bu}=(\cdots @>>>0@>>> 0@>>>\CK^0@>{\d^0}>>\CK^1@>{\d^1}>>\CK^2@>{\d^2}>>\cdots)\\
@VV{\phi^{\bu}}V @VVV@VV{\phi^0}V@VV{\phi^1}V@VV{\phi^2}V@VVV\\
\quad\qquad\P^{\bu}=(\cdots @>{\d^{-3}}>>\P^{-2}@>{\d^{-2}}>>\P^{-1}@>{\d^{-1}}>>\P^0@>{\d^0}>>\P^1@>{\d^1}>>\P^2@>{\d^2}>>\cdots)\\
@VVV@VVV @VV{}V@VV{}V@VV{}V@VVV\\
\quad\qquad\C^{\bu}=(\cdots @>{\d^{-3}}>>\C^{-2}@>{\d^{-2}}>>\C^{-1}@>{\widetilde{\d^{-1}}}>>\C^0@>{\widetilde{\d^0}}>>\C^1
@>{\widetilde{\d^1}}>>\C^2@>{\widetilde{\d^2}}>>\cdots)
\end{CD}$}\] where the complex $\mathcal{C}^{\bullet}={\rm Coker}(\phi^{\bullet})$ is given by
\begin{equation*}\C^i=\begin{cases} \bigoplus_{j\in Q_{0}}P_{j}^{({\La^{i, \dag}_{j}})},&\text{if $i\geq 0$};\\
\P^i, &\text{otherwise}.
\end{cases}
\end{equation*}
and 
\end{comment}

We observe the inclusions $\widetilde{\d}^l(\bigoplus_{i\in Q_{0}}P_{i}^{(\La^{l, n}_{i})})\subseteq\bigoplus_{i\in Q_{0}}P_{i}^{(\La^{l+1, n-1}_{i})}$ inside the complex $\C^{\bu}$ for each $l\in\Z$ and $n\geq 2$. Then for each $n\geq 0$ the following complex, denoted by $\C^{\bu}_n$, $$\cdots\stackrel{\d^{n-4}}{\longrightarrow}
\bigoplus_{i\in Q_{0}}P_{i}^{(\La^{n-3, 3}_{i})}\stackrel{\d^{n-3}}{\longrightarrow}\bigoplus_{i\in Q_{0}}P_{i}^{(\La^{n-2, 2}_{i})}\stackrel{\d^{n-2}}{\longrightarrow}
\bigoplus_{i\in Q_{0}}P_{i}^{(\La^{n-1, 1}_{i})}
\rightarrow 0$$ is a subcomplex of $\C^{\bu}$ satisfying $\C^l_n=0$ for $l\geq n$. The differential $\d^{l}$ for
$l\leq n-2$ is the differential of $\P^{\bullet}$.

We visually represent the projective Leavitt complex $\P^{\bu}$ and the cokernel complex $\C^{\bu}$ of $\phi^{\bu}$. For each $l\in\Z$ and $n\geq 0$, we denote $\bigoplus_{i\in Q_{0}}P_{i}^{(\La^{l, n}_{i})}$ by $P^{(\La^{l, n})}$ for simplicity.

\bigskip

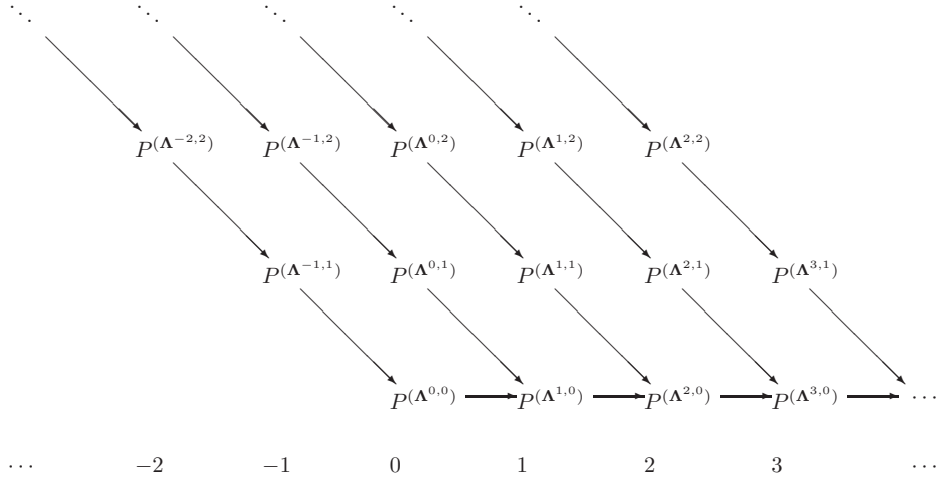
\[\resizebox{1.\hsize}{!}{$
$$~~~~~~~~~~~~~~~~~~~~~~~\setlength{\unitlength}{1mm}
\begin{picture}(150,90)

\put(0,81){$\ddots$}
\put(20,81){$\ddots$}
\put(40,81){$\ddots$}
\put(60,81){$\ddots$}
\put(80,81){$\ddots$}

\put(20,60){$P^{(\La^{-2, 2})}$}
\put(40,60){$P^{(\La^{-1, 2})}$} 
\put(60,60){$P^{(\La^{0, 2})}$}
\put(80,60){$P^{(\La^{1, 2})}$}
\put(100,60){$P^{(\La^{2, 2})}$}

\put(6,79){\vector (1,-1){15}}
\put(26,79){\vector (1,-1){15}}
\put(46,79){\vector (1,-1){15}}
\put(66,79){\vector (1,-1){15}}
\put(86,79){\vector (1,-1){15}}

\put(26,59){\vector (1,-1){15}}
\put(46,59){\vector (1,-1){15}}
\put(66,59){\vector (1,-1){15}}
\put(86,59){\vector (1,-1){15}}
\put(106,59){\vector (1,-1){15}}

\put(40,40){$P^{(\La^{-1, 1})}$}
\put(60,40){$P^{(\La^{0, 1})}$}
\put(80,40){$P^{(\La^{1, 1})}$}
\put(100,40){$P^{(\La^{2, 1})}$}
\put(120,40){$P^{(\La^{3, 1})}$}  

\put(46,39){\vector (1,-1){15}}
\put(66,39){\vector (1,-1){15}}
\put(86,39){\vector (1,-1){15}}
\put(106,39){\vector (1,-1){15}}
\put(126,39){\vector (1,-1){15}}

\put(60,20){$P^{(\La^{0, 0})}$}
\put(80,20){$P^{(\La^{1, 0})}$}
\put(100,20){$P^{(\La^{2, 0})}$} 
\put(120,20){$P^{(\La^{3, 0})}$}
\put(142,21){$\cdots$}

\put(72,22){\vector (1,0){8}}
\put(92,22){\vector (1,0){8}}
\put(112,22){\vector (1,0){8}}
\put(132,22){\vector (1,0){8}}

\put(142,10){$\cdots$}
\put(120,10){$3$}
\put(100, 10){$2$}
\put(80, 10){$1$}
\put(60, 10){$0$}
\put(40, 10){$-1$}
\put(20, 10){$-2$}
\put(0, 10){$\cdots$}
\end{picture}$$ $}\]

\bigskip

\begin{rmk} 
\begin{enumerate}
\item[(1)] For each $l\in\Z$, the $l$-th component of the projective Leavitt complex $\P^{\bu}$ is the coproduct of the objects in the $l$-th column of the above diagram. The differentials of $\P^{\bu}$ are coproducts of the maps in the diagram.

\item[(2)] The horizontal line of the above diagram is the subcomplex $\CK^{\bu}$, while the other part gives  the cokernel $\C^{\bu}$ of $\phi^{\bu}:\CK^{\bu}\xra \P^{\bu}$. The diagonal lines (not including the intersection with the horizontal line) of the above diagram are the subcomplexes $\C^{\bu}_n$ of $\C^{\bu}$. For example, the first diagonal line on the left of the above diagram (not including $P^{(\La^{0,0})}$) is the subcomplex $\C_0^{\bu}$. 
\end{enumerate}
\end{rmk}

We have the following observation immeadiately.

\begin{prop} \label{cokcop} The complex $\C^{\bu}=\bigoplus_{n\geq 0}\C^{\bu}_n$.
\end{prop}

\begin{proof} Observe that for each $n\geq 0$, the $l$-th component of $\C^{\bu}_n$ is
\begin{equation*}\C_n^l=\begin{cases}
\bigoplus_{i\in Q_0}P_i^{(\La^{l, n-l}_i)}, &\text{if $l<n$;}\\
0,& \text{otherwise}.
\end{cases}
\end{equation*} Then the $l$-th component of $\bigoplus_{n\geq 0}\C^{\bu}_n$ is
$\bigoplus_{n\geq 0}\C_n^l=
\bigoplus_{i\in Q_0}P_i^{(\La^{l,+}_i)}=\C^l$. Recall the differential $\widetilde{\d}^l:\C^l\xra \C^{l+1}$ of $\C^{\bullet}$. The restriction of $\widetilde{\d}^l$ to  $\bigoplus_{i\in Q_0}P_i^{(\La^{l,1}_i)}$ is zero and  the restriction of $\widetilde{\d}^l$ to $\bigoplus_{i\in Q_0}P_i^{(\La^{l,n}_i)}$ for $n>1$ is $\d^l$. Thus $\widetilde{\d}^l:\C^l\xra \C^{l+1}$ is the coproduct of the differentials $\d^l:\C^l_n\xra\C^{l+1}_n$ for $n\geq 0$. The proof is completed.
\end{proof}

\subsection{An explicit compact generator of the homotopy category} We denote by $A$-Mod the category of left $A$-modules. Denote by $\K(A\-\Mod)$ the homotopy category of $A$-Mod. We will always view a module as a stalk complex concentrated in degree zero.

For $X^{\bullet}=(X^{n}, d^n_{X})_{n\in \Z}$ a complex of $A$-modules, we denote by $X^{\bullet}[1]$ the complex given by $(X^{\bullet}[1])^{n}=X^{n+1}$ and
$d^n_{X[1]}=-d^{n+1}_{X}$ for $n\in\Z$.
For a chain map $f^{\bullet}: X^{\bullet}\xrightarrow[]{} Y^{\bullet}$ , its \emph{mapping cone} ${\rm Con}(f^{\bullet})$ is
a complex such that ${\rm Con}(f^{\bullet})=X^{\bullet}[1]\oplus Y^{\bullet}$ with the differential
$d^n_{{\rm Con}(f^{\bullet})}=\begin{pmatrix}-d^{n+1}_{X}&0\\f^{n+1}&d^n_{Y}\end{pmatrix}.$
Each triangle in $\K(A\-\Mod)$ is isomorphic to $$\CD
  X^{\bullet} @>f^{\bullet}>> Y^{\bullet}@>{\begin{pmatrix}0\\1\end{pmatrix}} >> {\rm Con}(f^{\bullet}) @>{\begin{pmatrix}1&0\end{pmatrix}}>> X^{\bullet}[1]
\endCD$$ for some chain map $f^{\bullet}$.

Denote by $I_{i}=D(e_{i}A_A)$ the injective left $A$-module for each vertex $i\in Q_{0}$, where $(e_{i}A)_{A}$ is the indecomposable projective right $A$-module and $D={\rm Hom}_{k}(-, k)$ denotes
the standard $k$-duality.
Denote by $\{e_{i}^{\s}\}\cup\{ \alpha^{\s}\;| \;\alpha\in Q_{1}, t(\alpha)=i\}$ the basis of
$I_i$, which is dual to the basis $\{e_{i}\}\cup\{
\alpha\;|\;\alpha\in Q_{1}, t(\alpha)=i\}$ of $e_iA$.

We denote by $\M^{\bullet}$ the following complex
$$
0\rightarrow\bigoplus_{i\in Q_{0}}I_{i}^{(\La^{0, 0}_{i})}
\stackrel{d^{0}}{\longrightarrow} \bigoplus_{i\in Q_{0}}I_{i}^{(\La^{1, 0}_{i})}
\stackrel{}{\longrightarrow}\cdots
\stackrel{}{\longrightarrow} \bigoplus_{i\in Q_{0}}I_{i}^{(\La^{l, 0}_{i})}
\stackrel{d^{l}}{\longrightarrow} \bigoplus_{i\in Q_{0}}I_{i}^{(\La^{l+1, 0}_{i})}
\stackrel{}{\longrightarrow} \cdots
$$ of $A$-modules satisfying $\M^l=0$ for $l<0$,  where the differential $d^{l}$ for
$l\geq 0$ is given by $d^{l}(e_{i}^{\s}\zeta_{(e_i, q)})=0$
and $d^{l}(\alpha^{\s}\zeta_{(e_i, q)})=e_{s(\alpha)}^{\s}\zeta_{(e_{s(\aa)}, q\alpha)}$ for
$i\in Q_{0}$, $(e_{i}, q)\in \La^{l, 0}_{i}$ and $\alpha\in Q_{1}$ with $t(\alpha)=i$. Consider the semisimple left $A$-module $kQ_0=A/\rad A$.

\begin{lem} \label{injres} The left $A$-module $kQ_0=A/\rad A$ is quasi-isomorphic to the complex $\M^{\bullet}$. In other words, $\M^{\bu}$ is an injective resolution of the $A$-module $kQ_0$.
\end{lem}

\begin{proof} Define a left $A$-module map $f^0: kQ_0\longrightarrow \M^0$ such that
$f^0(e_i)=e_i^{\s}\zeta_{(e_i, e_i)}$ for each $i\in Q_{0}$. Then we obtain a chain map $f^{\bullet}=(f^l)_{l\in\Z}: kQ_0
\longrightarrow \M^{\bullet}$ such that $f^l=0$ for $l\neq 0$. We observe the following $k$-basis of $\M^l$ for $l\geq 0$: $$\G^l=\{e_i^{\s}\zeta_{(e_i, q)}, \aa^{\s}\z_{(e_i, q)}\;|\;i\in Q_0, (e_i, q)\in \La^{l, 0}_i \text{~and~} \aa\in Q_1 \text{~with~} t(\aa)=i\}.$$
Set $\G^l_0=\{e_i^{\s}\zeta_{(e_i, q)}\;|\;i\in Q_0, (e_i, q)\in \La^{l, 0}_i\}$, $\G^l_1=\G^l\setminus \G^l_0$, and $\G'^l_1=\G^l$. We have the disjoint union $\G^l=\G^l_0\cup \G^l_1$.
The triple $(d^{l}, \G^l, \G^{l+1})$ satisfies Condition (W).
By Lemma $\ref{kerimproj1}$ the set $\G^l_0$ is a $k$-basis of $\Ke d^{l}$ and the
set $d^{l}(\G^l_1)$
is a $k$-basis of $\Im d^{l}$.
For each $l\geq 0$, $i\in Q_0$ and $(e_i, q)\in \La^{l+1, 0}_i$, write $q=\widetilde{q}\aa$ with $\aa\in Q_1$.
Then we have $e_i^{\s}\zeta_{(e_i, q)}=d^l(
\aa^{\s}\zeta_{(e_{t(\aa)}, \widetilde{q})}$.
Thus $d^{l}(\G^l_1)=\G^{l+1}_0$.
Hence $\Im d^{l}=\Ke d^{l+1}$ for each $l\geq 0$
and $\Ke d^0\cong kQ_0$. The statement follows directly. \end{proof}

%We now recall some terminology and facts on triangulated categories.
For a triangulated category $\mathcal{A}$, a \emph{thick} subcategory of $\mathcal{A}$
is a triangulated subcategory of $\mathcal{A}$ that is closed under direct summands. Let $\mathcal{S}$ be a class of objects in $\mathcal{A}$. Denote by thick$\langle\mathcal{S}\rangle$ the smallest thick subcategory of $\mathcal{A}$
containing $\mathcal{S}$.
If the triangulated category $\mathcal{A}$ has arbitrary coproducts, we denote by
${\rm Loc}\langle\mathcal{S}\rangle$ the smallest triangulated subcategory of
$\mathcal{A}$ which contains $\mathcal{S}$ and is closed under arbitrary coproducts.
By \cite[Proposition 3.2]{bbn}, thick$\langle\mathcal{S}\rangle\subseteq{\rm Loc}\langle \mathcal{S}\rangle$.

For a triangulated category $\mathcal{A}$ with arbitrary coproducts, we say that an object $M\in\mathcal{A}$ is \emph{compact} if the functor ${\rm Hom}_{\mathcal{A}}(M,-)$ commutes with arbitrary coproducts.
Denote by $\mathcal{A}^{c}$ the full subcategory consisting of compact objects; it is a thick subcategory.

A triangulated category $\mathcal{A}$ with arbitrary coproducts is \emph{compactly generated} \cite{bke, bn1} if there exists a set $\mathcal{S}$ of compact objects such that
any nonzero object $X$ satisfies that ${\rm Hom}_{\mathcal{A}}(S,X[n])\neq 0$
for some $S\in\mathcal{S}$ and $n\in\Z$.  This  is equivalent to the condition that
$\mathcal{A}={\rm Loc}\langle\mathcal{S}\rangle$, in which case $\mathcal{A}^{c}$=thick$\langle\mathcal{S}\rangle$; see \cite[Lemma 3.2]{bn1}. If the above set $\mathcal{S}$
consists of a single object $S$, we call $S$ a \emph{compact generator} of the triangulated category $\mathcal{A}$.

The following is \cite[Lemma 3.9]{bli}.

\begin{lem} \label{comgproj}Suppose that $\mathcal{A}$ is a compactly generated triangulated category
with a compact generator $X$. Let $\mathcal{A}'\subseteq \mathcal{A}$ be a triangulated
subcategory which is closed under arbitrary coproducts. Suppose that there exists a triangle
\begin{align*}\CD
 X @>>>Y@>>>Z @>>>X[1]
\endCD
\end{align*} such that $Y\in \mathcal{A}'$ and $Z$ satisfies $\Hom_{\mathcal{A}}(Z, A')=0$
for each object $A'\in\mathcal{A}'$. Then $Y$ is a compact generator of $\mathcal{A}'$.\hfill $\square$
\end{lem}

Let $A\-\Inj$ and $A\-\Proj$ be the categories of injective and projective $A$-modules, respectively. Denote by $\K(A\-\Inj)$ and $\K(A\-\Proj)$ the homotopy categories
of complexes of injective and projecitve $A$-modules, respectively. These homotopy categories are triangulated subcategories of $\K(A\-\Mod)$ which are closed under coproducts. By \cite[Proposition 2.3(1)]{bkr} $\K(A\-\Inj)$ is a compactly generated triangulated category.

Recall that the Nakayama functor $\nu=DA\otimes_A-:A\-\Proj\lra A\-\Inj$ is an equivalence, whose quasi-inverse
$\nu^{-1}={\rm Hom}_A(D(A_A), -)$. Thus we have a triangle equivalence $\K(A\-\Inj)\xrightarrow[]{\sim}\K(A\-\Proj)$. The category $\K(A\-\Proj)$ is a compactly generated triangulated category; see \cite[Theorem 2.4]{bj} and \cite[Proposition 7.14]{bn3}.

\begin{lem}\label{comgenproj} The complex $\CK^{\bu}$ is a compact generator of $\K(A\-\Proj)$.
\end{lem}

\begin{proof} We first prove that $\CK^{\bu}\cong(\nu^{-1}(\M^{i}), \nu^{-1}(d^i))_{i\in \Z}$ as complexes. Recall that $$\nu^{-1}(I_i)=\Hom_A(D(A_A), I_i)\cong Ae_i$$ for each $i\in Q_0$. Observe that all the sets $\La^{l,0}_i$ are finite. Then for each $l\geq 0$ we have the isomorphism of $A$-modules $$f^l:\bigoplus_{i\in Q_0} P_i^{(\La^{l,0}_i)}\stackrel{\sim}\longrightarrow\Hom_A(D(A_A), \bigoplus_{i\in Q_0}I_i^{(\La^{l, 0}_i)})$$ such that $f^l(e_i\z_{(e_i, q)})(e_j^{\s})=\d_{ij}e_j^{\s}\z_{(e_i, q)}$ and $f^l(e_i\z_{(e_i, q)})(\b^{\s})=\d_{i, t(\b)}\b^{\s}\z_{(e_i, q)}$ for each $i, j\in Q_0$, $(e_i, q)\in\La^{l,0}_i$ and $\b\in Q_1$. We have $\Hom_A(D(A_A), d^l)\circ f^l=f^{l+1}\circ \d^l$ by direct calculation for each $l\geq 0$.

Recall from Lemma \ref{injres} that $\M^{\bu}$ is an injective resolution of the left $A$-module $kQ_0$. It follows from \cite[Proposition 2.3]{bkr} that
$\M^{\bu}$ is a compact
object in $\K(A\-\Inj)$ and ${\rm Loc}\langle\mathcal{\M^{\bu}}\rangle=\K(A\-\Inj)$.
Since $\K(A\-\Inj)\xrightarrow[]{\sim}\K(A\-\Proj)$ is a triangle equivalence which sends $\M^{\bu}$ to $\CK^{\bu}$,
we have ${\rm Loc}\langle\mathcal{\CK^{\bu}}\rangle=\K(A\-\Proj)$.
\end{proof}

\begin{lem} \label{lemmazero}Suppose that $\mathrm{X}^{\bu}\in \mathbf{K}(A\-\Proj)$ is a bounded-above complex.
Then we have $${\rm Hom}_{\mathbf{K}(A\-\Mod)}(\mathrm{X}^{\bullet}, Y^{\bullet})=0$$ for any acyclic complex $Y^{\bullet}$
of $A$-modules.
\end{lem}

\begin{proof} Directly check that any chain map $f^{\bu}:\mathrm{X}^{\bullet}\lra Y^{\bu}$ is null-homotopic.
\end{proof}

Denote by ${\mathbf{K}_{\rm ac}}(A\-\Proj)$ the full subcategory of $\mathbf{K}(A\-\Mod)$ which is
formed by acyclic complexes of projective $A$-modules. Applying \cite[Propositions 7.14 and 7.12]{bn3}
and the localization theorem in \cite[1.5]{bke2}, we have that
the category is a compactly generated triangulated category
with the triangle equivalence
$${\rm \mathbf{D}_{sg}}(A^{\rm op})^{\rm op}\stackrel{\sim}\longrightarrow{\mathbf{K}_{\rm ac}}(A\-\Proj)^c.$$
Here, for a category $\mathcal{C}$, we denote by $\mathcal{C}^{\rm op}$ its opposite category; the category ${\rm \mathbf{D}_{sg}}(A^{\rm op})$ is the singularity category of algebra $A^{\rm op}$ in the sense of \cite{bbu, bo}.

\begin{thm}\label{tcproj} Let $Q$ be a finite quiver without sources. Then the projective Leavitt complex $\mathcal{P}^{\bullet}$ of $Q$
is a compact generator of the homotopy category ${\rm \mathbf{K}_{ac}}(A\-\Proj)$.
\end{thm}
\begin{proof} Recall from Proposition \ref{projacy}  that $\mathcal{P}^{\bullet}$ is an object of ${\rm \mathbf{K}_{ac}}(A\-\Proj)$.
The complex $\mathcal{C}^{\bullet}={\rm Coker}(\phi^{\bullet})$, where $\phi^{\bullet}: \mathcal{K}^{\bullet}\longrightarrow
\mathcal{P}^{\bullet}$ is the inclusion chain map. Then we have the following exact sequence 
$$\xymatrix@C=0.5cm{
  0 \ar[r] & \mathcal{K}^{\bullet} \ar[rr]^{\phi^{\bullet}} && \mathcal{P}^{\bullet} \ar[rr]^{} && \mathcal{C}^{\bullet}  \ar[r] & 0 ,}$$
 which splits in each component.
This gives rise to a triangle \begin{align}\label{eq:q}\CD
 \mathcal{K}^{\bullet}@>\phi^{\bullet}>>\mathcal{P}^{\bullet}@>>>\mathcal{C}^{\bullet} @>>>X[1]
\endCD
\end{align} in the category $\mathbf{K}(A\-\Proj)$. 

By Proposition \ref{cokcop} and Lemma \ref{lemmazero}, the following equality holds \begin{equation*}\Hom_{\mathbf{K}(A\-\Proj)}(\mathcal{C}^{\bullet}, X^{\bullet})= \prod_{n\geq 0}{\Hom}_{\mathbf{K}(A\-\Proj)}(\mathcal{C}^{\bullet}_n, X^{\bullet})=0\end{equation*} for
any $X^{\bullet}\in \mathbf{K}_{\rm ac}(A\-\Proj)$. Recall from Lemma \ref{comgenproj} that $\CK^{\bu}$ is a compact generator of $\K(A\-\Proj)$. By the triangle \eqref{eq:q} and Lemma \ref{comgproj}, the proof is completed.
%$\mathcal{P}^{\bullet}$ is a compact generator of the category ${\rm \mathbf{K}_{ac}}(A\-\Proj)$.
\end{proof}

\section{The projective Leavitt complex as a differential graded bimodule}
\label{section4}

In this section, we endow the projective Leavitt complex with a differential graded bimodule structure over the corresponding Leavitt path algebra.

\subsection{A module structure of the Leavitt path algebra} Let $k$ be a field and $Q$ be a finite quiver.
We will endow the projective Leavitt complex of $Q$ with a Leavitt path algebra module structure.
Recall from \cite{bap, bamp}
the notion of the Leavitt path algebra.

\begin{defi} \label{defleavittproj}
The \emph{Leavitt path algebra} $L_{k}(Q)$ of $Q$
is the $k$-algebra generated by the set $\{e_{i}\;|\;i\in Q_{0}\}\cup \{\alpha\;|\;\alpha\in Q_{1}\}
\cup\{\alpha^{*}\;|\;\alpha\in Q_{1}\}$ subject to the following relations:

(0) $e_{i}e_{j}=\delta_{i, j}e_{i}$ for every $i, j\in Q_{0}$;

(1) $e_{t(\alpha)}\alpha=\alpha e_{s(\alpha)}=\alpha$ for all $\alpha\in Q_{1}$;

(2) $e_{s(\alpha)}\alpha^{*}=\alpha^{*} e_{t(\alpha)}=\alpha^{*}$ for all $\alpha\in Q_{1}$;

(3) $\alpha\beta^{*}=\delta_{\alpha, \beta}e_{t(\alpha)}$ for all $\alpha,\beta\in Q_{1}$;

(4) $\sum_{\{\alpha\in Q_{1}\;|\;s(\alpha)=i\}}\alpha^{*}\alpha=e_{i}$ for $i\in Q_{0}$ which is not a sink.\hfill $\square$
\end{defi}

Here, $\delta$ denotes the Kronecker symbol. The above relations $(3)$ and $(4)$ are called
\emph{Cuntz-Krieger relations}. The elements $\alpha^{*}$ for $\alpha\in Q_{1}$ are called \emph{ghost arrows}.

The Leavitt path algebra $L_k(Q)$ can be viewed as a quotient algebra of the path algebra as follows.
Denote $\overline{Q}$ the \emph{double quiver} obtained from $Q$ by adding an arrow $\alpha^{*}$ in the opposite direction for each arrow $\alpha$ in $Q$.
The Leavitt path algebra $L_{k}(Q)$ is isomorphic to the quotient algebra of the path algebra $k\overline{Q}$ of the double quiver $\overline{Q}$
modulo the ideal generated by $\{\alpha\beta^{*}-\delta_{\alpha, \beta}e_{t(\alpha)},
\sum_{\{\g\in Q_{1}\;|\; s(\g)=i\}}\g^{*}\g-e_{i}\;|\;\alpha,\beta\in Q_{1}, i\in Q_{0}\text{~which is not a sink}\}.$

\begin{comment}
By the relation $(2)$, we have that for paths $p, q$ in $Q$, $p^*q=0$ if $t(p)\neq t(q)$.
Consider the relation $(3)$. The following observation follows directly; see \cite[Lemma 3.1]{bt}.

\begin{lem}\label{mul}Let $p$, $q$, $\g$ and
$\eta$ be paths in $Q$ with $t(p)=t(q)$ and $t(\g)=t(\eta)$.
Then in the Leavitt path algebra $L_k(Q)$ we have
\begin{equation*}
(p^*q)(\g^*\eta)=
\begin{cases}
(\g' p)^*\eta, & \text{if $\g=\g'q$};\\
p^*\eta, & \text{if $q=\g$};\\
p^*(q'\eta), & \text{if $q=q'\g$};\\
0, &\text{otherwise}.
\end{cases}
\end{equation*} Here, $\g'$ and $q'$ are some nontrivial paths in $Q$.\hfill $\square$
\end{lem}
\end{comment}

If $p=\alpha_{m}\cdots\aa_2\alpha_{1}$ is a path in $Q$ of length $m\geq 1$, we define $p^{*}=\alpha_{1}^*\aa_2^*\cdots\alpha_{m}^*$.
For convention, we set $e_{i}^*=e_{i}$ for $i\in Q_{0}$. The Leavitt path algebra $L_k(Q)$ is spanned by the following set
$\{p^*q\;|\; p, q \text{~are paths in~} Q \text{~with~} t(p)=t(q)\};$
see \cite[Lemma 1.5]{bap}, \cite[Corollary 3.2]{bt} or \cite[Corollary 2.2]{bc2}. By the relation $(4)$, this set is not $k$-linearly independent in general.

For each vertex which is not a sink, we fix a special arrow starting at it.
The following result is \cite[Theorem 1]{baajz}.

\begin{lem} \label{lbasisproj} The following elements form a $k$-basis of the Leavitt path algebra $L_k(Q)$:\begin{enumerate}\item[(1)] $e_i$, $i\in Q_0$;

\item[(2)]  $p, p^*$, where $p$ is a nontrivial path in $Q$;

\item[(3)]  $p^*q$ with $t(p)=t(q)$, where
$p=\alpha_{m}\cdots\alpha_{1}$ and $q=\beta_{n}\cdots\beta_{1}$ are nontrivial paths of $Q$ such that $\alpha_{m}\neq \beta_{n}$, or $\alpha_{m}=\beta_{n}$ which is not special.\hfill $\square$\end{enumerate}
\end{lem}

From now on, $Q$ is a finite quiver without sources. For notation, $Q^{\rm op}$
is the opposite quiver of $Q$. For a path $p$ in $Q$, denote by
$p^{\rm op}$ the corresponding path in $Q^{\rm op}$. The starting and terminating vertices of $p^{\rm op}$
are $t(p)$ and $s(p)$, respectively.
For convention, $e_{j}^{\rm op}=e_{j}$ for each vertex $j\in Q_{0}$. The opposite quiver $Q^{\rm op}$ has no sinks.

For the opposite quiver $Q^{\rm op}$ of $Q$, choose $\aa^{\rm op}$ to
be the special arrow of $Q^{\rm op}$ starting at vertex $i$, where $\aa$ is the associated arrow in $Q$ terminating at $i$. By Lemma \ref{lbasisproj}
there exists a $k$-basis of the Leavitt path algebra $L_k(Q^{\rm op})$, denoted by $\G$. Define a map $\chi:\bigcup_{l\in\Z, i\in Q_{0}}\La^l_i \xra\G$ such that 
$\chi(p, q)={(p^{\rm op})}^*q^{\rm op}$. Here, ${(p^{\rm op})}^*q^{\rm op}$ is the multiplication of ${(p^{\rm op})}^*$ and $q^{\rm op}$ in $L_k(Q^{\rm op})$. The map $\chi$ is a bijection. We identify $\G$ with the set of associated pairs in $Q$. A nonzero element $x$ in $L_k(Q^{\rm op})$ can be written in the unique form
\begin{equation*}
x=\sum_{i=1}^{m}\lambda_{i}(p_{i}^{\rm op})^{\ast}q_{i}^{\rm op}
\end{equation*}
with $\lambda_{i}\in k$ nonzero scalars
and $(p_{i}, q_{i})$ pairwise distinct associated pairs in $Q$.

In what follows, $B=L_k(Q^{\rm op})$. We write $ab$ for the multiplication of $a$ and $b$ in B for $a, b\in B$. Recall that the projective Leavitt complex
$\mathcal{P}^{\bullet}=(\P^{l}, \d^l)_{l\in\Z}$ and
$\mathcal{P}^{l}=\bigoplus_{i\in Q_{0}}{P_{i}}^{(\La^{l}_{i})}$.

We define a right $B$-module action on $\P^{\bullet}$.
For each vertex $j\in Q_{0}$ and each arrow $\alpha\in Q_{1}$, define right actions $``{\cdot}"$ on
$\mathcal{P}^{l}$ for any $l\in\Z$ as follows. 
For any element $x\zeta_{(p, q)}\in P_{i}\zeta_{(p, q)}$ with $i\in Q_{0}$ and
$(p, q)\in \mathbf{\La}^{l}_{i}$, we set 

\begin{equation}\label{actionproj1}{x\zeta_{(p, q)}}{\cdot} e_{j}=\delta_{j, t(q)}x\zeta_{(p, q)};\end{equation}
\begin{equation}\label{actionproj2}{x\zeta_{(p, q)}}{\cdot}\alpha^{\rm op}=\begin{cases}
 x\zeta_{(\widetilde{p}, e_{t(\alpha)})}-\sum\limits_{\beta\in T(\alpha)}x
 \zeta_{(\widetilde{p}\beta, \beta)}
 , & \begin{matrix}
\text{if~} l(q)=0, p=\widetilde{p}\alpha,\\ \text{and~}\alpha \text{~is associated};
\end{matrix}\\
 \delta_{s(\alpha),t(q)}x\zeta_{(p, \alpha q)}, & \text{otherwise}.
\end{cases}
\end{equation}
\begin{equation}\label{actionproj3}{x\zeta_{(p, q)}}{\cdot} (\alpha^{\rm op})^{*}=\begin{cases}
\delta_{\alpha, \alpha_{1}}x\zeta_{(p, \widehat{q})}, & \text{if $q=\alpha_{1}\widehat{q}$ };\\
\delta_{s(p),t(\alpha)}x\zeta_{(p\alpha,  e_{s(\alpha)})}, & \text{if $l(q)=0$}.
\end{cases}
\end{equation}

Here for the notation, a path $p=\alpha_{m}\cdots\aa_2\alpha_{1}$ in $Q$ of length
$m\geq 2$ has two truncations $\widehat{p}=\alpha_{m-1}\cdots\alpha_{1}$ and $\widetilde{p}=\alpha_{m}\cdots\alpha_{2}$. For an arrow $\aa$, $\widehat{\aa}=e_{s(\aa)}$ and $\widetilde{\aa}=e_{t(\aa)}$. The set $T(\aa)=\{\b\in Q_1\;|\; t(\b)=t(\aa), \b\neq \aa\}$ for an associated arrow $\aa$.

We observe the following fact:
\begin{equation}
\label{obsactionproj}
\begin{cases}
{x\zeta_{(p, q)}}{\cdot} \alpha^{\rm op}=0, & \text{If $s(\alpha)\neq t(q)$};\\
{x\zeta_{(p, q)}}{\cdot} (\alpha^{\rm op})^{*}=0, &\text{If $t(\alpha)\neq t(q)$}.
 \end{cases}
\end{equation}

\begin{lem} \label{lmoduleproj} The above actions make the projective Leavitt complex $\mathcal{P}^{\bullet}$ of $Q$
a right $B$-module.
\end{lem}
\begin{proof} We prove that the above right actions
satisfy the defining relations of the Leavitt path algebra $L_k(Q^{\rm op})$ of the opposite quiver $Q^{\rm op}$. We fix $x\z_{(p, q)}\in P_i\z_{(p, q)}\subseteq \P^l$.

For $(0)$, we observe that ${x\zeta_{(p, q)}}{\cdot} (e_{j}\circ e_{j'})=\delta_{j, j'}
{x\zeta_{(p, q)}}{\cdot} e_{j}.$

For $(1)$, for each $\alpha\in Q_{1}$ we have 
\begin{align*}{x\zeta_{(p, q)}}{\cdot}(\alpha^{\rm op} e_{t(\alpha)})
&=({x\zeta_{(p, q)}}{\cdot}\alpha^{\rm op}){\cdot} e_{t(\alpha)}\\
&={x\zeta_{(p, q)}}{\cdot}\alpha^{\rm op}\end{align*} 
We have
\begin{align*}{x\zeta_{(p, q)}}{\cdot}( e_{s(\alpha)}\alpha^{\rm op})
&=({x\zeta_{(p, q)}}{\cdot}e_{s(\alpha)}){\cdot}\alpha^{\rm op}\\
&=\delta_{s(\aa), t(q)}{x\zeta_{(p, q)}}{\cdot}\alpha^{\rm op}\\
&={x\zeta_{(p, q)}}_{\cdot}\alpha^{\rm op},\end{align*} where the last equality uses \eqref{obsactionproj}.
Similar arguments prove the relation $(2)$.

For $(3)$, we have that for $\alpha, \beta\in Q_{1}$
\[
\begin{aligned}
&{x\zeta_{(p, q)}}{\cdot} (\alpha^{\rm op}(\beta^{\rm op})^{*})=({x\zeta_{(p, q)}}{\cdot}\alpha^{\rm op}){\cdot}(\beta^{\rm op})^{*}\\
&=\left\{\begin{array}{ll}
\delta_{t(\alpha), t(\beta)}x\zeta_{(\widetilde{p}\beta, e_{s(\beta)})}-
\sum\limits_{\gamma\in T(\alpha)}
\delta_{\gamma, \beta}x\zeta_{(\widetilde{p}\gamma, e_{s(\gamma)})}, &
\begin{matrix}\text{if~~} l(q)=0, p=\widetilde{p}\alpha\\ \text{and~}\alpha
\text{~is~associated;} \end{matrix}\\
 \delta_{s(\alpha), t(q)}\delta_{\alpha,\beta}x\zeta_{(p, q)}, & \text{otherwise};
 \end{array}\right.\\
&=\delta_{s(\alpha), t(q)}\delta_{\alpha,\beta}x\zeta_{(p, q)}\\
&={x\zeta_{(p, q)}}{\cdot}(\delta_{\alpha, \beta}e_{s(\alpha)})
\end{aligned}
\]
Here, we use the fact that
in the case that $l(q)=0$, $p=\widetilde{p}\alpha$ and $\alpha$ is associated, if $\alpha=\beta$,
then $s(\alpha)=t(q)$ and $\gamma\neq\beta$ for each $\gamma\in T(\alpha)$; and if
$\alpha\neq\beta$ with $t(\alpha)=t(\beta)$, then there exists an arrow $\gamma\in T(\aa)$ such that $\gamma=\beta$.

For $(4)$, for each $j\in Q_{0}$ we have that: if $\alpha\in Q_{1}$ with $t(\alpha)=j$ is associated, then
\begin{equation*}\begin{split}{x\zeta_{(p, q)}}{\cdot}( (\alpha^{\rm op})^{*}\alpha^{\rm op})
&=({x\zeta_{(p, q)}}{\cdot}(\alpha^{\rm op})^{*}){\cdot}\alpha^{\rm op}\\
&=\begin{cases}
\delta_{\alpha, \alpha_{1}}(x\zeta_{(p,  q)}, & \text{if $q=\alpha_{1}\widehat{q}$};\\
\delta_{j',s(p)}(x\zeta_{(p, e_{s(p)})}-\sum\limits_{\beta\in
T(\alpha)}x\zeta_{(p\beta, \beta)}), & \text{if $l(q)=0$}.\\
\end{cases}
\end{split}
\end{equation*} 
If $\alpha\in Q_{1}$ with $t(\alpha)=j$ is not associated, then
\begin{equation*}{x\zeta_{(p, q)}}{\cdot}((\alpha^{\rm op})^{*}\alpha^{\rm op})=\begin{cases}
\delta_{\alpha, \alpha_{1}}x\zeta_{(p,  q)}, & \text{if $q=\alpha_{1}\widehat{q}$};\\
\delta_{j, s(p)}x\zeta_{(p\alpha, \alpha)}, & \text{if $l(q)=0$}.\\
\end{cases}
\end{equation*} Thus, we have the following equality
\[
\begin{aligned}
&{x\zeta_{(p, q)}}{\cdot}(\sum_{\{\alpha\in Q_{1}\;|\;t(\alpha)=j\}}(\alpha^{\rm op})^{*}\alpha^{\rm op})\\
&=\left\{\begin{array}{ll}
\delta_{j, t(q)}x\zeta_{(p, q)}, & \text
{if $q=\alpha_{1}\widehat{q}$};\\
 \delta_{j, s(p)}x\zeta_{(p, e_{s(p)})}, & \text{if $l(q)=0$}.
 \end{array}\right.\\
&=\delta_{j, t(q)}x\zeta_{(p, q)}\\
&={x\zeta_{(p, q)}}{\cdot} e_{j}.
\end{aligned}
\] 
\end{proof}

The following observation gives an intuitive description of the $B$-module action on $\P^{\bu}$.

\begin{lem} \label{lmodulep} Let $(p, q)$ be an associated pair in $Q$.
\begin{enumerate}
\item We have $\sum_{i\in Q_0}e_i\zeta_{(e_{i}, e_i)}\cdot (p^{\rm op})^*q^{\rm op}=e_{t(p)}\zeta_{(p,  q)};$ 

\item For each arrow $\b\in Q_1$, we have the following equality holds: $$\b\zeta_{(e_{s(\b)}, e_{s(\b)})}\cdot (p^{\rm op})^*q^{\rm op}=\d_{s(\b), t(p)}\b\zeta_{(p, q)}.$$
\end{enumerate}
\end{lem}
\begin{proof} Since $(p, q)$ is an associated pair in $Q$, we are in the second subcases in \eqref{actionproj3} and \eqref{actionproj2} for the right action of $(p^{\rm op})^*q^{\rm op}$. Then the statements follow from direct calculation.
\end{proof}

\subsection{A differential graded bimodule}

We recall some notion on differential graded modules; see \cite{bke}.
Let $A=\bigoplus_{n\in\Z}A^{n}$ be a $\Z$-graded algebra. For a (left) graded $A$-module $X=\bigoplus_{n\in\Z}X^n$, elements $x$ in $X^{n}$ are said to be homogeneous of degree $n$, denoted by $|x|=n$.

A \emph{differential graded algebra} (dg algebra for short) is a $\Z$-graded algebra $A$ with a differential
$d:A \xra A$ of degree one satisfying $d(ab)=d(a)b+(-1)^{|a|}ad(b)$ for homogenous elements
$a,b\in A$.

A \emph{(left) differential graded} $A$-module (dg $A$-module for short) $X$
is a graded $A$-module $X=\bigoplus_{n\in \mathbb{Z}}X^{n}$ with a differential $d_{X}:X\xra X$
of degree one satisfying $d_{X}(a{\cdot} x)=d(a){\cdot} x+(-1)^{|a|}a{\cdot} d_{X}(x)$ for homogenous
elements $a\in A$ and $x\in X$. A morphism of dg $A$-modules is a morphism of $A$-modules which 
preserves degrees and commutes with differentials.
A \emph{right differential graded} $A$-module (right dg $A$-module for short) $Y$
is a right graded $A$-module $Y=\bigoplus_{n\in \Z}Y^{n}$ with a differential $d_{Y}:Y\xra Y$
of degree one satisfying $d_{Y}(y\cdot a)=d_{Y}(y)\cdot a+(-1)^{|y|}y\cdot d(a)$ for homogenous
elements $a\in A$ and $y\in Y$. Here, central dots denotes the $A$-module action.

Let $B$ be another dg algebra.
Recall that a \emph{dg $A$-$B$-bimodule} $X$ is a left dg $A$-module as well as a right dg
$B$-module satisfying $(a\cdot x)\cdot b=a\cdot (x\cdot b)$ for $a\in A$, $x\in X$ and $b\in B$.

Recall that $Q$ is a finite quiver without sources. In what follows, we write $B=L_k(Q^{
\rm op})$, which is naturally $\Z$-graded by the length of paths. We view $B$ as a dg algebra with trivial differential.

Consider $A=kQ/J^{2}$ as a dg algebra concentrated in degree zero. 
Recall the projective Leavitt complex $\P^{\bu}=\bigoplus_{l\in\Z}\P^l$. It is a left dg $A$-module.
By Lemma \ref{lmoduleproj}, $\P^{\bu}$ is a right $B$-module. By (\ref{actionproj1}),
(\ref{actionproj2}) and (\ref{actionproj3}), we have that $\P^{\bu}$ is a right graded $B$-module.

There is a unique right $B$-module morphism $\phi: B\xra \P^{\bullet}$ with $$\phi(1)=\sum_{i\in Q_0}e_i\z_{(e_i, e_i)}.$$ Here, $1$ is the unit of $B$. For each arrow $\b\in Q_1$, there is a unique right $B$-module morphism $\phi_{\b}:B\xra \P^{\bu}$ with $\phi_{\b}(1)=\b\z_{(e_{s(\b)}, e_{s(\b)})}$. By Lemma \ref{lmodulep} we have 
\begin{equation}
\label{defphi}
\phi((p^{\rm op})^*q^{\rm op}) 
=e_{t(p)}\z_{(p, q)}\text{~and~} \phi_{\b}((p^{\rm op})^*q^{\rm op}) 
=\delta_{s(\b), t(p)}\b\z_{(p, q)}
\end{equation} 
for $(p^{\rm op})^*q^{\rm op}\in\G$.
Here, we emphasize that $\G$ is the $k$-basis of $B=L_k(Q^{\rm op})$. Then $\phi$ is injective and the restriction of $\phi_{\b}$ to $e_{s(\b)}B$ is injective. Observe that both $\phi$ and $\phi_{\b}$ are graded $B$-module morphisms. 

\begin{lem} \label{deltamodule} For each $i\in Q_0$, $l\in \Z$ and $(p, q)\in\La^l_i$, we have $$(\d^{l}\circ\phi)((p^{\rm op})^{*}q^{\rm op})
=\sum_{\{\aa\in Q_1\;|\;t(\aa)=i\}}\phi_{\aa}(\aa^{\rm op}(p^{\rm op})^{*}q^{\rm op}).$$ From this, we have that $(\d^{l}\circ\phi)(b)
=\sum_{\aa\in Q_1}\phi_{\aa}(\aa^{\rm op}b)$ for $b\in B^l$.
\end{lem}

\begin{proof} For each arrow $\aa\in Q_1$ and $(p^{\rm op})^*q^{\rm op}\in\G$, we observe that 
\begin{equation}
\label{eq:aaa}
\aa^{\rm op}(p^{\rm op})^*q^{\rm op}=
\begin{cases}
\delta_{\aa, \aa_1}(\widehat{p}^{\rm op})^{*}q^{\rm op},& \text{if $p=\aa_1 \widehat{p}$};\\
(q\aa)^{\rm op},& \text{if $l(p)=0$},
\end{cases}
\end{equation} which are combinations of basis elements of $L_k(Q^{\rm op})$. Then we have that \begin{equation*}
\begin{split}(\d^{l}\circ\phi)((p^{\rm op})^{*}q^{\rm op})
&=\d^{l}(e_i\z_{(p, q)})\\
&=\begin{cases}
\aa_1\z_{(\widehat{p},q)},& \text{if $p=\aa_1\widehat{p}$};\\
\sum_{\{\aa\in Q_1\;|\;t(\aa)=i\}}\aa\z_{(e_{s(\aa)}, q\aa)}, & \text{if $l(p)=0$}.
\end{cases}\\
&=\sum_{\{\aa\in Q_1\;|\;t(\aa)=i\}}\phi_{\aa}(\aa^{\rm op} (p^{\rm op})^*q^{\rm op}).
\end{split}
\end{equation*} Here, the last equality uses \eqref{eq:aaa}.
\end{proof}

It is evident that the projective Leavitt complex $\P^{\bu}$ is a graded $A$-$B$-bimodule. The following result shows that $\P^{\bu}$ is a dg $A$-$B$-bimodule.  Recall from Definition \ref{defproj} for the differentials $\d^l$ of $\P^{\bu}$.

\begin{prop} \label{propbimproj}  For each $l\in\Z$, let $x\zeta_{(p, q)}\in P_i\zeta_{(p, q)}$ with $i\in Q_{0}$
and $(p, q)\in \La^{l}_{i}$. Then for each vertex $j\in Q_{0}$ and each arrow $\beta\in Q_{1}$ we have 

\begin{enumerate}

\item $\d^{l}({x\zeta_{(p, q)}}{\cdot}e_{i})=\d^{l}(x\zeta_{(p, q)}){\cdot}e_{i};$

\item $ \d^{l+1}({x\zeta_{(p, q)}}{\cdot}\beta^{\rm op})=\d^{l}(x\zeta_{(p, q)}){\cdot}\beta^{\rm op};$

\item $\d^{l-1}({x\zeta_{(p, q)}}{\cdot}(\beta^{\rm op})^{*})=\d^{l}(x
\zeta_{(p, q)}){\cdot}(\beta^{\rm op})^{*}.$
\end{enumerate} In other words, the right $B$-action makes $\P^{\bu}$ a right dg $B$-module and thus a dg $A$-$B$-bimodule.
\end{prop}
\begin{proof}
Recall that $\d^l(\aa\zeta_{(p, q)})=0$ for $\aa\in Q_1$ with $s(\aa)=i$.
It follows that $(1)$, $(2)$ and $(3)$ hold for $x=\aa$. It suffices to prove that $(1)$, $(2)$ and $(3)$ hold for $x=e_i$. We recall that $(p, q)\in \La^l_i$, and thus $t(p)=i$.

For $(1)$, we have that
\begin{equation*}
\begin{split}
\d^{l}({e_i\zeta_{(p, q)}}{\cdot}e_{j})
&=\d^{l}(\phi((p^{\rm op})^*q^{\rm op})e_j)\\
&=(\d^{l}\circ\phi)((p^{\rm op})^*q^{\rm op}e_j)\\
%&=\delta_{j, t(q)}\sum_{\{\aa\in Q_1, t(\aa)=i\}}\psi_{\aa}(\aa^{\rm op}(p^{\rm op})^*q^{\rm op})\\
&=\sum_{\{\aa\in Q_1\;|\;t(\aa)=i\}}\phi_{\aa}(\aa^{\rm op}(p^{\rm op})^*q^{\rm op}e_{j})\\
&=\sum_{\{\aa\in Q_1\;|\;t(\aa)=i\}}\phi_{\aa}(\aa^{\rm op}(p^{\rm op})^*q^{\rm op}){\cdot}e_{j}\\
&=\d^{l}(e_i\zeta_{(p, q)}){\cdot}e_{j}.
\end{split}
\end{equation*} Here, the second and the fourth equalities hold because $\phi$ and $\phi_{\aa}$ are right $B$-module morphisms; the third and the last equalities use Lemma \ref{deltamodule};  Similar arguments prove $(2)$ and $(3)$.
\end{proof}

\section{The differential graded endomorphism algebra of the projective Leavitt complex}
\label{section5}

In this section, we prove that the opposite differential graded endomorphism algebra of the projective Leavitt complex of a finite quiver without sources is quasi-isomorphic to the Leavitt path algebra of the opposite quiver.  Here, the Leavitt path algebra is naturally $\Z$-graded and viewed as a differential graded algebra with trivial differential.

\subsection{The quasi-balanced dg bimodule}
\label{subsection51}

We first recall some notion and facts about quasi-balanced dg bimodules.
Let $A$ be a dg algebra and $X, Y$ be (left) dg $A$-modules.
We have a $\Z$-graded vector space ${\rm Hom}_{A}(X, Y)=\bigoplus_{n\in\Z}{\rm Hom}_{A}(X, Y)^{n}$
such that each component ${\rm Hom}_{A}(X, Y)^{n}$ consists of $k$-linear maps
$f:X\xra Y$ satisfying $f(X^i)\subseteq Y^{i+n}$ for all $i\in\Z$ and $f(a\cdot x)=(-1)^{n|a|}a
\cdot f(x)$ for all homogenous elements $a\in A$.
The differential on ${\rm Hom}_{A}(X, Y)$ sends $f\in
{\rm Hom}_{A}(X, Y)^{n}$ to $d_{Y}\circ f-(-1)^n f\circ d_{X}\in
{\rm Hom}_{A}(X, Y)^{n+1}$.
Furthermore, ${\rm End}_{A}(X):={\rm Hom}_{A}(X, X)$ becomes a dg algebra
with this differential and the usual composition as multiplication.
The dg algebra ${\rm End}_{A}(X)$ is usually called the \emph{dg endomorphism algebra} of the dg module $X$.

We denote by $A^{\rm opp}$ the \emph{opposite dg algebra} of a dg algebra $A$, that is,
$A^{\rm opp}=A$ as graded spaces with the same differential, and the multiplication $``\circ"$
on $A^{\rm opp}$ is given such that $a\circ b=(-1)^{|a||b|}ba$.

Let $B$ be another dg algebra.
Recall that a right dg $B$-module is a left dg $B^{\rm opp}$-module.
For a dg $A$-$B$-bimodule $X$, the canonical map
$A\xra {\rm End}_{B^{\rm opp}}(X)$ is a homomorphism of dg algebras, sending $a$ to $l_{a}$ with $ l_{a}(x)=a\cdot x$ for $a\in A$ and $x\in X$.
Similarly, the canonical map
$B\xra {\rm End}_{A}(X)^{\rm opp}$ is a homomorphism of dg algebras,
sending $b$ to $ r_{b}$ with $ r_{b}(x)=(-1)^{|b||x|}x\cdot b$ for homogenous elements $b\in B$ and $x\in X$.

A dg $A$-$B$-bimodule $X$ is called \emph{right quasi-balanced} provided that the canonical
homomorphism $B \xra {\rm End}_{A}(X)^{\rm opp}$ of dg algebras
is a quasi-isomorphism; see \cite[2.2]{bcy}.

We denote by $\K(A)$ the homotopy category of left dg $A$-modules and by $\D(A)$ the derived category
of left dg $A$-modules; they are both triangulated categories with arbitrary coproducts.
For a dg $A$-$B$-bimodule
$X$ and a left dg $A$-module $Y$, ${\rm Hom}_{A}(X, Y)$ has a natural structure of left dg $B$-module.

The following lemma is \cite[Proposition 2.2]{bcy}; compare \cite[4.3]{bke} and \cite[Appendix A]{bkr}.

\begin{lem} \label{equivproj} Let $X$ be a dg $A$-$B$-bimodule which is right quasi-balanced.
Recall that ${\rm Loc}\langle X\rangle\subseteq\K(A)$ is the smallest triangulated
subcategory of
$\K(A)$ which contains $X$ and is closed under arbitrary coproducts. Assume that
$X$ is a compact object in ${\rm Loc}\langle X\rangle$.
Then we have a triangle equivalence \begin{equation*}{\rm Hom}_{A}(X,-):{\rm Loc}\langle X\rangle \stackrel{\sim}\longrightarrow \D(B).\tag*{$\square$}\end{equation*}
\end{lem}

In what follows, $Q$ is a finite quiver without sources and $A=kQ/J^{2}$ is the corresponding algebra with radical square zero.
Consider $A$ as a dg algebra concentrated in degree zero.
Recall that the Leavitt path algebra $B=L_k(Q^{\rm op})$ is naturally $\Z$-graded, and that it is viewed as
a dg algebra with trivial differential.

Recall from Proposition \ref{propbimproj} that the projective Leavitt complex $\P^{\bullet}$ is a dg $A$-$B$-bimodule. There is a connection between the projective Leavitt complex and the Leavitt path algebra, which is established by the following statement.

\begin{thm} \label{rightqproj} Let $Q$ be a finite quiver without sources.  Then the dg $A$-$B$-bimodule
$\mathcal{P}^{\bullet}$ is right quasi-balanded.

In particular, the opposite dg endomorphism algebra of the projective Leavitt complex of $Q$ is quasi-isomorphic to the Leavitt path algebra $L_k(Q^{\rm op})$ of $Q^{\rm op}$.  Here, $Q^{\rm op}$ is the opposite quiver of $Q$; $L_k(Q^{\rm op})$ is naturally $\Z$-graded and viewed as a dg algebra with trivial differential.
\end{thm}

We will prove Theorem \ref{rightqproj} in subsection \ref{subsection52}.
The following equivalence was given by \cite[Theorem 6.2]{bcy}. 

\begin{cor} Let $Q$ be a finite quiver without sources. Then there is a triangle
equivalence $${\rm Hom}_{A}(\P^{\bullet},-):\K_{\rm ac}(A\-\Proj) \stackrel{\sim}\longrightarrow\D(B)$$
such that ${\rm Hom}_{A}(\P^{\bullet}, \P^{\bullet})\cong B$ in $\D(B)$.
%Restricting the equivalence to compact objects, we obtain the triangle equivalence $\D_{\rm sg}(A^{\rm op})^{\rm op}\stackrel{\sim}\longrightarrow\D(B)^{c}$.
\end{cor}

\begin{proof} Recall from Theorem \ref{tcproj} that $\K_{\rm ac}(A\-\Proj)={\rm Loc}\langle \P^{\bullet}\rangle$. Then the triangle equivalence follows from Theorem \ref{rightqproj} and Lemma \ref{equivproj}. The canonical map $B\xra {\rm End}_A(\P^{\bu})^{\rm opp}$ is a quasi-isomorphism, which identifies ${\rm Hom}_{A}(\P^{\bullet}, \P^{\bullet})$ with $B$ in $\D(B)$.
\end{proof}

\subsection{The proof of Theroem \ref{rightqproj}}
\label{subsection52}

In this subsection, we follow the notation in subsection \ref{subsection51}.

\begin{lem} $($\cite[Theorem 4.8]{bt}$)$ \label{uniquethmproj}
Let $A$ be a $\Z$-graded algebra and $\eta:L_k(Q)\xra A$ be a graded algebra homomorphism such that 
$\eta(e_{i})\neq 0$ for all $i\in Q_{0}$. Then $\eta$ is injective.\hfill $\square$
\end{lem}

Let $Z^{n}$ and $C^{n}$ denote the $n$-th cocycle and
coboundary of the dg algebra ${\rm End}_{A}(\mathcal{P}^{\bullet})^{\rm opp}$. We have the following observation:

\begin{lem}\label{coboundaryproj} Any element $f:\mathcal{P}^{\bullet}\longrightarrow \mathcal{P}^{\bullet}$ in $C^{n}$
satisfies $f(\mathcal{P}^{l})\subseteq \Ke\delta^{n+l}$ for each integer $l$ and $n$.
\end{lem}
\begin{proof} For any $f\in C^{n}$ there exists
$h=(h^l)_{l\in\Z}\in {\rm End}_{A}(\mathcal{P}^{\bullet})^{\rm opp}$ such that $f^l=\delta^{n+l-1}\circ h^l-(-1)^{n-1}h^{l+1}\circ\delta^{l}$ for each $l\in\Z$.
By Proposition \ref{projacy} we have $\Im \d^{n+l-1}\subseteq \Ke\d^{n+l}$ and $\Im \d^l\subseteq \Ke\d^{l+1}$. Then it suffices to prove $h(\Ke\delta^{l+1})\subseteq \Ke\delta^{n+l}$. Recall from Lemma \ref{deltaker} that $\{\alpha\zeta_{(p, q)}\;|\; i\in Q_{0}, (p, q)\in\mathbf{\La}^{l+1}_{i} \text{~and~}\aa\in Q_1 \text{~with~}s(\aa)=i \}$ is a $k$-basis of $\Ke\d^{l+1}$. By Lemma \ref{lemma-projmap}, we are done.
\end{proof}

Recall that the projective Leavitt complex $\P^{\bullet}$ is a dg $A$-$B$-bimodule; see Proposition \ref{propbimproj}.
Let $\rho: B\xra {\rm End}_{A}(\mathcal{P}^{\bullet})^{\rm opp}$ be the canonical map which is
induced by the right $B$-action. We have that $\rho(L_k(Q^{\rm op})^{n})\subseteq Z^{n}$ for each $n\in \Z$, since $B$ is a dg algebra with trivial differential. Taking cohomologies, the following graded algebra homomorphism is obtained
\begin{equation}
\label{ccc}
H(\rho):
B\longrightarrow H({\rm End}_{A}(\mathcal{P}^{\bullet})^{\rm opp}).
\end{equation}

\begin{lem} \label{embeddingproj} The graded algebra homomorphism $H(\rho)$ is an embedding.
\end{lem}

\begin{proof} By Lemma \ref{uniquethmproj}, it suffices to prove that $H(\rho)(e_i)\neq 0$ for all $i\in Q_0$.
For each vertex $i\in Q_0$, $H(\rho)(e_i)(e_i\zeta_{(e_i, e_{i})})=e_i\zeta_{(e_i, e_{i})}$. By Lemma \ref{deltaker}
we have $e_i\zeta_{(e_i, e_{i})}\notin \Ke\d^{0}$.
By Lemma \ref{coboundaryproj}, $H(\rho)(e_i)\notin C^0$. This implies $H(\rho)(e_i)\neq 0$.
\end{proof}

It remains to prove that the graded algebra homomorphism $H(\rho)$ is surjective.
For each element $y\in Z^n$, we will find an element $x\in B^n$ such that $y-\rho(x)\in C^n$.

In what follows, we fix $y\in Z^{n}$ for $n\in\Z$. Then $y\in Z^{n}$ implies $\delta^{\bu}\circ y-(-1)^{n}y\circ\delta^{\bu}=0$. Recall that $\P^l=\bigoplus_{i\in Q_0}P_i^{(\La^l_i)}$ for each $l\in\Z$. The set $\{e_i\z_{(p, q)}, \aa\z_{(p, q)}\; |\; i\in Q_0, (p, q)\in \La^l_i \text{~and~} \aa\in Q_1 \text{~with~} s(\aa)=i\}$ is a $k$-basis of $\P^l$. For each $i\in Q_{0}$, $l\in\Z$, 
and $(p, q)\in\mathbf{\Lambda}^{l}_{i}$, we have
\begin{equation}
\label{eq:p}
\begin{cases}
(\delta^{n+l}\circ y)(e_{i}\zeta_{(p, q)})=(-1)^{n}(y\circ \delta^l)(e_{i}\zeta_{(p, q)}) & \\
(\delta^{n+l}\circ y)(\alpha\zeta_{(p, q)})=0,&
\end{cases}
\end{equation} where $\alpha \in Q_{1}$ with $s(\aa)=i$.

Observe that $y$ is an $A$-module morphism. By Lemma \ref{lemma-projmap}, we may assume that
\begin{equation}
\label{eq:nu}
\begin{cases}
y(e_{i}\zeta_{(p, q)})=\phi(y_{(p, q)})+\sum_{\{\gamma\in Q_{1}\;|\; t(\g)=i\}}\phi_{\gamma}(\mu^{\gamma}_{(p, q)})&\\
y(\alpha\zeta_{(p, q)})=\phi_{\alpha}(y_{(p, q)}),
\end{cases}
\end{equation} where $y_{(p, q)}\in e_iL_k(Q^{\rm op})^{n+l}$ and
$\mu^{\gamma}_{(p, q)}\in e_{s(\g)}L_k(Q^{\rm op})^{n+l}$.

By \eqref{eq:nu} and Lemma \ref{deltamodule}, we have that $$(\delta^{n+l}\circ y)(e_{i}\zeta_{(p, q)})=\d^{n+l}(\phi(y_{(p, q)}))=\sum_{\{\g\in Q_1\;|\; t(\g)=i\}}\phi_{\g}(\g^{\rm op}y_{(p, q)})$$ and that
$$(y\circ \delta^{l})(e_{i}\zeta_{(p, q)})=y(\sum_{\{\g\in Q_1\;|\; t(\g)=i\}}\phi_{\g}(\g^{\rm op}(p^{\rm op})^*q^{\rm op})).$$
By $(\ref{eq:p})$, we have \begin{equation}
\label{eq:bothside}
\sum_{\{\g\in Q_1\;|\;  t(\g)=i\}}\phi_{\g}(\g^{\rm op}y_{(p, q)})
=(-1)^ny(\sum_{\{\g\in Q_1\;|\; t(\g)=i\}}\phi_{\g}(\g^{\rm op}(p^{\rm op})^*q^{\rm op})).
\end{equation}

We recall that for each arrow $\b\in Q_1$, the restriction of $\phi_{\b}$ to $e_{s(\b)}B$ is injective. If $l(p)=0$, then by (\ref{eq:bothside}) and \eqref{eq:nu} for any $\gamma\in Q_{1}$ with $t(\g)=i$ we have 
\begin{equation}
\label{eq:x1}
y_{(e_{s(\gamma)}, q\gamma)}=(-1)^{n}\gamma^{\rm op} y_{(p, q)}.
\end{equation} If $l(p)>0$, write $p=a\widehat{p}$ with $a\in Q_1$ and $t(a)=i$. By (\ref{eq:bothside}) and \eqref{eq:nu} we have
$a^{\rm op}y_{(p, q)}=(-1)^{n}y_{(\widehat{p}, q)}$ and $\gamma^{\rm op} y_{(p, q)}=0$ for $\gamma\in Q_1$ with 
$t(\g)=i$ and
$\gamma\neq a$. The following equality holds:
\begin{equation}
\label{eq:x2}
y_{(p, q)}=\sum_{\{\g\in Q_1\;|\; t(\g)=i\}} (\g^{\rm op})^*\g^{\rm op} y_{(p,q)}=(-1)^{n}(a^{\rm op})^{*} y_{(\widehat{p}, q)}.
\end{equation}

\begin{lem} \label{ruleproj} Keep the notation as above. Take $x=\sum_{j\in Q_{0}}y_{(e_{j}, e_{j})}\in L_k(Q^{\rm op})^{n}$. Then $y_{(p, q)}=(-1)^{nl}(p^{\rm op})^{*}q^{\rm op} x$ in $L_k(Q^{\rm op})$ for each $(p,q)\in \La^l_i$.
\end{lem}

\begin{proof} Clearly, we have $y_{(e_{i}, e_{i})}=e_{i} x$ in $L_k(Q^{\rm op})$.
For $l(p)=0$ and $l(q)>0$, write $q=\widetilde{q}\gamma$ with $\g\in Q_1$.
We use \eqref{eq:x1} to obtain $y_{(e_{s(q)}, q)}=(-1)^{nl}q^{\rm op}x$ in $L_k(Q^{\rm op})$ by induction on $l(q)$. For $l(p)>0$, write $p =\beta_{m}\cdots\beta_{1}$ with all $\b_k$ arrows in $Q$. By $(\ref{eq:x2})$ we have
$y_{(p, q)}=(-1)^{n}(\beta_{m}^{\rm op})^{*}y_{(\widehat{p}, q)}$.
We obtain $y_{(p, q)}=(-1)^{nm}(\beta_{m}^{\rm op})^{*}\cdots(\beta_{1}^{\rm op})^{*}y_{(e_{s(q)}, q)}$ by iterating \eqref{eq:x2}. Then we are done by $y_{(e_{s(q)},q)}=(-1)^{n(m+l)}q^{\rm op}x$ in $L_k(Q^{\rm op})$.
\end{proof}

We will construct a map $h:\P^{\bu}\xra \P^{\bu}$ of degree $n-1$, which will be used to prove $y-\rho(x)\in C^n$. To define $h$, we first assign to each pair $(p,q)\in\La^l_i$ an element $\theta_{(p,q)}$ in $L_k(Q^{\rm op})$.

For each $i\in Q_0$, define $\theta_{(e_{i}, e_{i})}
=\sum_{\{\gamma\in Q_{1}\;|\; s(\g)=i\}}\mu^{\gamma}_{(\gamma, e_{s(\gamma)})}\in e_iL_k(Q^{\rm op})^{n-1}$.
Here, refer to (\ref{eq:nu}) for the element $\mu^{\gamma}_{(\gamma, e_{s(\gamma)})}$. We define $\theta_{(e_{s(q)}, q)}$ inductively by 
\begin{equation}
\label{eq:ac1}
\theta_{(e_{s(q)}, q)}=(-1)^{n-1}(\g^{\rm op}\theta_{(e_{s(\widetilde{q})}, \widetilde{q})}
-\mu^{\g}_{(e_{s(\widetilde{q})}, \widetilde{q})})\in e_{s(q)}L_k(Q^{\rm op})^{n+l-1},\end{equation} where $q=\widetilde{q}\g$ with $l(q)=l$ and $\g\in Q_1$. Let $(p,q)\in\La^l_i$ with $l(p)>0$. We define $\theta_{(p, q)}$ by induction on the length of $p$ as follows: \begin{equation}
\label{eq:ac2}
\theta_{(p, q)}=(-1)^{n-1}(\beta^{\rm op})^{*}\theta_{(\widehat{p}, q)}+\sum_{\{\gamma\in Q_{1}\;|\; t(\g)=i\}}(\gamma^{\rm op})^{*}\mu^{\gamma}_{(p, q)}\in e_{i}L_k(Q^{\rm op})^{n+l-1},
\end{equation} where $p=\beta\widehat{p}$ with $\b\in Q_1$ is of length $l(q)-l$.

We define a $k$-linear map $h:\P^{\bu}\lra\P^{\bu}$ such that 
\begin{equation*}
h(e_{i}\zeta_{(p, q)})=\phi(\theta_{(p, q)}) ~~\text{and}~~h(\alpha\zeta_{(p, q)})=\phi_{\alpha}(\theta_{(p, q)})
\end{equation*} for each $i\in Q_{0}$, $l\in\Z$, $(p, q)\in\mathbf{\La}^{l}_{i}$,
and $\alpha\in Q_{1}$ with $s(\aa)=i$.

\begin{lem} \label{last}Let $x$ be the element in Lemma \ref{ruleproj} and $h$ be the above map. For each $i\in Q_0$, $l\in\Z$, $(p, q)\in\mathbf{\La}^{l}_{i}$, we have 
\begin{equation*}
\begin{cases} 
(y-\rho(x))(e_{i}\zeta_{(p, q)})=(\delta^{n+l-1} \circ h-(-1)^{n-1}h\circ\delta^l)(e_{i}\zeta_{(p, q)})&\\
(y-\rho(x))(\alpha\zeta_{(p, q)})=(\delta^{n+l-1} \circ h-(-1)^{n-1}h\circ\delta^l)(\alpha\zeta_{(p, q)})=0,
\end{cases}
\end{equation*} where $\alpha\in Q_{1}$ with $s(\aa)=i$.
\end{lem}

\begin{proof} Recall from \eqref{defphi} the right $B$-module morphisms $\phi$ and $\phi_{\b}$ for $\b\in Q_1$. By \eqref{eq:nu} and Lemma \ref{ruleproj}, we have \begin{equation*}
\begin{cases}
\rho(x)(e_{i}\zeta_{(p, q)})
=(-1)^{nl}\phi({p^{\rm op}}^{*}q^{\rm op}){\cdot} x=\phi(y_{(p, q)})&\\
\rho(x)(\alpha\zeta_{(p, q)})=(-1)^{nl}\phi_{\aa}({p^{\rm op}}^{*}q^{\rm op}){\cdot} x
=\phi_{\aa}(y_{(p, q)})&\\
(y-\rho(x))(e_{i}\zeta_{(p, q)})
=\sum_{\{\gamma\in Q_{1}\;|\; t(\g)=i\}}\phi_{\gamma}(\mu^{\gamma}_{(p, q)});
\end{cases}
\end{equation*} Recall that $\d^l\circ \phi_{\b}=0$ for each arrow $\b\in Q_1$. It remains to prove $(\delta^{n+l-1} \circ h-(-1)^{n-1}h\circ\delta^l)(e_{i}\zeta_{(p, q)})=\sum_{\{\gamma\in Q_{1}\;|\; t(\g)=i\}}\phi_{\gamma}(\mu^{\gamma}_{(p, q)})$.

By the definition of $\d^l$, we have
\begin{equation*}(h\circ\delta^l)(e_{i}\zeta_{(p, q)})=\begin{cases}
 \phi_{\beta}(\theta_{(\widehat{p}, q)}), & \text{if $p=\beta\widehat{p}$};\\
 \sum_{\{\gamma\in Q_{1}\;|\;t(\g)=i\}}\phi_{\gamma}(\theta_{(e_{s(\gamma)}, q\gamma)}), & \text{if $l(p)=0$}.
\end{cases}
\end{equation*} By Lemma \ref{deltamodule}, we have 
$(\delta^{n+l-1} \circ h)(e_{i}\zeta_{(p, q)})=\sum_{\{\gamma\in Q_{1}\;|\; t(\g)=i\}}\phi_{\gamma}(\gamma^{\rm op}\theta_{(p, q)})$. Then the following equalities hold: 
\begin{equation*}\begin{split}&(\delta^{n+l-1} \circ h)(e_{i}\zeta_{(p, q)})
-(-1)^{n-1}(h\circ\delta^l)(e_{i}\zeta_{(p, q)})\\
&=\begin{cases}
 \sum_{\{\gamma\in Q_{1}\;|\;t(\g)=i\}}\phi_{\gamma}(\gamma^{\rm op}\theta_{(p, q)})-(-1)^{n-1}\phi_{\beta}(\theta_{(\widehat{p}, q)}),
 & \text{if $p=\beta\widehat{p}$};\\
 \sum_{\{\gamma\in Q_{1}\;|\;t(\g)=i\}}\phi_{\gamma}(\gamma^{\rm op}\theta_{(p, q)}-(-1)^{n-1}\theta_{(e_{s(\gamma)}, q\gamma)}), & \text{if $l(p)=0$}
\end{cases}\\
&=\sum_{\{\gamma\in Q_{1}\;|\; t(\g)=i\}}\phi_{\gamma}(\mu^{\gamma}_{(p, q)}).\end{split}
\end{equation*} Here, the last equality uses $(\ref{eq:ac1})$ and $(\ref{eq:ac2})$. 
\end{proof}

\vskip 10pt

\noindent{\emph{Proof of Theorem \ref{rightqproj}.}} It suffices to prove that $H(\rho)$ in (\ref{ccc}) is an isomorphism. By Lemma \ref{embeddingproj}, it remains to prove that $H^{n}(\rho)$ is surjective for any $n\in\Z$.
For any element $\overline{y}=y+C^{n}$ with $y\in Z^{n}$,
take $x                                                                                                                                                     =\sum_{j\in Q_{0}}y_{(e_{j}, e_{j})}\in B^n=L_k(Q^{\rm op})^{n}$.
By Lemma \ref{last}, we have $y-\rho(x)\in C^{n}$.
Then it follows that $\overline{y}=\overline{\rho(x)}$ in $H^n({\rm End}_A(\P^{\bu})^{\rm opp})$. \hfill $\square$

\section{Acknowledgements} The author thanks Professor Xiao-Wu Chen for inspiring discussions. The author thanks Australian Research Council grant DP160101481.

\vskip 10pt

{\footnotesize\noindent Huanhuan Li \\
Centre for Research in Mathematics, Western Sydney University, Australia\\
E-mail: h.li@westernsydney.edu.au}

\end{document}